\documentclass[a4paper,11pt]{article}
\def\vartheta{Z}
\usepackage[utf8]{inputenc}
\usepackage{graphicx}
\usepackage{amsmath,amsfonts,amsthm,amssymb}

\newtheorem{theorem}{Theorem}[section]
\newtheorem{lemma}[theorem]{Lemma}
\def\Id{\mathrm{Id}}
\def\diag{\mathrm{diag}}
\def\eps{\varepsilon}
\numberwithin{equation}{section}

\def\calH{\mathcal{H}}

\def\calM{\mathcal{M}}

\def\N{\mathbb{N}}
\def\R{\mathbb{R}}
\def\Z{\mathbb{Z}}
\DeclareMathOperator{\dist}{dist}
\DeclareMathOperator{\SO}{SO}
\DeclareMathOperator{\OO}{O}
\newcommand{\vii}[2]{\ensuremath{\begin{pmatrix}#1 \\ #2 \end{pmatrix}}}
\newcommand{\mii}[4]{\ensuremath{\begin{pmatrix}#1&#2 \\ #3&#4 \end{pmatrix}}}
\newcommand{\norm}[1]{\ensuremath{|| #1 ||}}

\begin{document}
\begin{center}
{ \LARGE
Energy scaling and branched microstructures in a model for
  shape-memory   alloys with   $\SO(2)$ invariance
 \\[5mm]}
{October 20, 2014}\\[5mm]
Allan Chan and Sergio Conti\\
{\em  Institut f\"ur Angewandte Mathematik,
Universit\"at Bonn\\ 53115 Bonn, Germany }
\\[3mm]
\begin{minipage}[c]{0.8\textwidth}\small
Domain branching near the boundary  appears in many
singularly perturbed models for microstructure in materials and was first
demonstrated  mathematically
by Kohn and Müller for a scalar problem representing the elastic behavior of
shape-memory alloys. 
We study here a model for shape-memory alloys based on
the full vectorial problem of nonlinear elasticity, including
 invariance under rotations, in the case of two wells in two dimensions. 
We show that, for two wells with two rank-one connections, the energy
is proportional to the power $2/3$ of the surface energy, in agreement
with the scalar model. In a case where only one rank-one
connection is present, we show that the energy exhibits a different behavior, 
proportional to the power $4/5$ of the surface energy. This lower energy is
achieved by a suitable interaction of the two components of the deformations
and hence cannot be reproduced by the scalar model. Both scalings are proven
by explicit constructions and matching lower bounds.
\end{minipage}
\end{center}


\section{Introduction}
The peculiar elastic properties of shape-memory alloys have attracted a large
interest in the mechanics literature, as discussed for example in \cite{B9_Bhatta,OtsukaWayman1999},
starting with the works 
of Ball and James \cite{B9_BallJames87,B9_BallJames92} also in the vectorial
calculus of variations, see also \cite{ChipotKinderlehrer88,Dacorogna1989,MuellerLectureNotes} and the 
references therein.  
The typical model, based on nonlinear elasticity, takes the form 
\begin{equation}\label{eqfctnonlineaelast}
\int_\Omega W(Du) dx\,,
\end{equation}
where $u:\Omega\to\R^n$ is the elastic deformation. The energy density 
 $W:\R^{n\times n}\to\R$  is
minimized by several copies of the set of proper rotations $\SO(n)$, which
correspond to the several phases of the material. The macroscopic material
behavior can be analyzed using the theory of relaxation, which -- under
appropriate continuity and growth conditions -- leads to the study of the
quasiconvex envelope of $W$ \cite{Dacorogna1989,MuellerLectureNotes}. Partial
results in this direction have been obtained for example in   
\cite{B9_BallJames92,Sverak1993,DolzmannLN,DolzmannKirchheim2003,ContiDolzmannKirchheim2007,ContiDolzmanntwowell}. 
The theory of relaxation does not, however, make 
any prediction on the length scale of the microstructure, as  the
functional (\ref{eqfctnonlineaelast}) is scale invariant. 

A more precise analysis of the microstructure is possible if one includes a
singular perturbation, typically of the form $\varepsilon|D^2u|(\Omega)$ (in the
sense of the total variation of the gradient of the $BV$ function $Du$) or
$\varepsilon^2 \|D^2u\|_{L^2}^2$, where $\varepsilon$ represents a (scaled)
surface energy. A scalar simplification of this type of model  was
proposed by Kohn and Müller in 1992
\cite{B9_KohnMuller92,B9_KohnMuller94}. They have shown that the minimum 
energy scales as $\varepsilon^{2/3}$, and that this scaling is achieved by a
self-similar microstructure, which refines close to the boundary. Finer results
were later obtained in \cite{B9_Conti00,B9_Conti2006,B9_Zwicknagl}, including
in particular the statement of 
asymptotic self-similarity for the minimizers of the Kohn-Müller functional.
Their approach was extended to many other physical problems, including for
example micromagnetism \cite{CK98,CKO99,Viehmanndiss}, magnetic structures in
superconductors \cite{B9_CKO2004,B9_CCKO2008}, dislocation structures in
plasticity \cite{B9_ContiOrtiz05}, coarsening in thin film growth
\cite{ContiOrtiz2008}. 

We consider in this work the full vectorial problem of nonlinear elasticity with
surface energy, namely, 
 the functional
\begin{equation}\label{intro:1}
  \int_{\Omega}W(Du)\,dx + \varepsilon |D^2u|(\Omega)\,.
\end{equation}
Here, $u:\Omega\subset\R^n\rightarrow\R^n$ represents the elastic deformation,
$W:\R^{n\times n}\rightarrow\R$ the energy density, which is minimized on a
set  $K\subset\R^{n\times n}$ and grows quadratically. For definiteness, we
shall focus on
\begin{equation}\label{intro:2}
 W(F) = \dist^2(F,K)= \inf\{|F-G|^2: \ G\in K\}\,,
\end{equation}
although all results hold for any energy with the same set of minimizers and
the same growth. The small positive parameter $\varepsilon$ in
(\ref{intro:1}) represents the surface energy, and the admissible 
deformations are those $u\in
W^{1,2}(\Omega;\R^n)$ with $Du\in BV(\Omega;\R^{n\times n})$ which obey
specific boundary conditions. By $|D^2u|(\Omega)$ we denote the total
variation of the measure $D^2u$ over the set $\Omega$.
The set $K\subset\R^{n\times n}$ of energy-minimizing deformation
gradients, as used in (\ref{intro:2}), 
depends on the crystallography of the  phase transformation 
considered. 

In this paper we focus on the two-dimensional case, $n=2$, and study 
 two different crystallographic settings.
The first situation we address, denoted case $K_1$, is 
\begin{equation}\label{assumption:K1}
 A_1= \mii{1}{-\alpha}{0}{1}, \quad B_1=\mii{1}{\alpha}{0}{1}, \quad K_1=\SO(2)A_1\cup\SO(2)B_1
\end{equation}
for some $\alpha>0$.
These matrices have the same determinant and are representative of
volume-preserving phase transformations in which the two wells have two rank-one
connections, in the sense that $\det(A_1-QB_1)=0$ has two solutions $Q\in
\SO(2)$. This is the vectorial situation that was represented by
the  scalar model of Kohn and Müller; indeed, one can obtain a model similar
to the one they used if 
$K_1$ is reduced to the two matrices $\{A_1, B_1\}$, see
\cite{B9_Zwicknagl}. The two components of $u$ are in this case 
decoupled, only the first one gives a nontrivial energy contribution. 

We show that for this problem the energy is  proportional to
$\varepsilon^{2/3}$ if the aspect ratio of
the domain is of order 1 (Theorem \ref{Theorem:K1}),
thereby extending the Kohn-Müller result to the present vectorial model. 
The presence of a second rank-one
connection becomes, however, important in the case of long and thin domains, in
which microstructure 
is generated along vertical oscillations instead of horizontal ones, as we
discuss after Theorem \ref{Theorem:K1} and in Section \ref{secubtheo1} below.

The second situation we consider, denoted case $K_2$, is
\begin{equation}\label{assumption:K2}
A_2=\mii{1}{0}{0}{1-\alpha}, \quad B_2=\mii{1}{0}{0}{1+\alpha}, \quad K_2=\SO(2)A_2\cup\SO(2)B_2
\end{equation}
for some $\alpha\in(0,1)$. For $\alpha=1$, the matrix $A_2$ becomes singular; 
treating this unphysical case would generate some complications in the proofs,
for clarity we restrict in the relevant parts to $\alpha<1/2$. Clearly, any
other condition $\alpha<\alpha_*$, for any fixed $\alpha_*<1$, would lead to the same results. 
Physically relevant values of the spontaneous strain are typically of a few percent, up to about $10\%$ \cite{B9_Bhatta,OtsukaWayman1999}, and therefore well within the  range $(0,1/2)$.

The matrices in (\ref{assumption:K2}) have a different determinant and are representative of 
phase transformations with one rank-one
connection, in the sense that $\det(A_2-QB_2)=0$ has a unique solution $Q\in 
\SO(2)$.  In this case, the invariance under rotations and the vectorial nature
of the deformation  
can interact to reduce the energy cost of domain branching. 
The
resulting scaling, proportional to
$\varepsilon^{4/5}$, differs from the one derived for the
 scalar problem. This difference
can be understood as a consequence of the degeneracy of the rank-one
connection (in the same sense as for  double roots of polynomials). 
To see this, let us consider a typical branching pattern, as illustrated in the top right part of Figure \ref{fig:PD2}, which consists of
 a fine-scale mixture of regions with $Du$ close to $\SO(2)A_j$ and regions with $Du$ close to $\SO(2)B_j$, with interfaces
approximately horizontal.
Since $u$ does not jump across the interface, if $u$ is sufficiently regular, 
the tangent component of the gradient does not jump either. In other words, given a vector  $v\in S^1$  tangent to the interface, then the directional derivative
$Duv$ must be the same on the two sides of the interface. 
If $Du\in \SO(2)A_j$ on one side, and $Du\in \SO(2)B_j$ on the other side, then
necessarily $|A_jv|=|B_jv|$. For $j=1$, this equation is equivalent to
\begin{equation*}
  (v_1-\alpha v_2)^2+v_2^2 = 
  (v_1+\alpha v_2)^2+v_2^2 \,,
\end{equation*}
which simplifies to $\alpha v_1v_2=0$. For $j=2$ the
condition $|A_jv|=|B_jv|$  is equivalent to
\begin{equation*}
  v_1^2 + (1-\alpha)^2 v_2^2
=  v_1^2 + (1+\alpha)^2 v_2^2\,,
\end{equation*}
which simplifies to $\alpha v_2^2=0$. It is easy to see that $v_2=0$ (corresponding to $v=e_1$) is a simple solution in the first case, but a degenerate solution in the second case. Correspondingly, 
\begin{equation*}
\text{$|A_1v|-|B_1v|=O(\alpha|v-e_1|)$ but $|A_2v|-|B_2v|=O(\alpha|v-e_1|^2)$.  }
\end{equation*}
 Therefore, small deviations of the interfacial direction from $v=e_1$ generate less elastic energy in the case $K_2$ than in the case $K_1$, leading to a lower global energy.
If rotations are neglected, which is the case if one treats the scalar case,
then the two terms are of the same order,
\begin{equation*}
A_1v-B_1v=-2\alpha v_2 e_1=O(\alpha |v-e_1|)
\end{equation*}
and
\begin{equation*}
A_2v-B_2v=-2\alpha v_2 e_2=O(\alpha |v-e_1|)\,.
\end{equation*}
Therefore, the difference between $K_1$ and $K_2$ is not present in scalar models.
The detailed exploitation of this effect in the construction requires an appropriate interaction between the two components and 
 will be discussed in Section \ref{secUBtheo2}, see in particular Lemma \ref{eqconstronerecttheo2}.

From a mechanical viewpoint, in most materials the phases are symmetry-related and in particular one expects the eigenstrains to have the same determinant 
\cite{B9_BallJames87,B9_Bhatta,OtsukaWayman1999}, 
which corresponds to our case $j=1$. 
The two-dimensional geometry we discuss corresponds  to a three-dimensional setting in which both eigendeformations coincide in the third dimension.
This would be,  for example, the case for the three-dimensional eigenstrains
\begin{equation*}
 A_1'= 
 \begin{pmatrix}
   1 & -\alpha & 0\\
 0 & 1 & 0 \\
0 & 0 &\gamma
 \end{pmatrix} \text{ and }
 \quad B_1'=
 \begin{pmatrix}
   1 & \alpha & 0\\
 0 & 1 & 0 \\
0 & 0& \gamma
 \end{pmatrix}
\end{equation*}
for some $\gamma>0$, the three-dimensional deformation would then be
of the form $\varphi(x_1,x_2,x_3)=(u_1(x_1,x_2),u_2(x_1,x_2),\gamma x_3)$.

In thin films, however, a different relation between the two- and the three-dimensional problem is possible. Indeed, if the film is sufficiently thin, then incompatibility in the transverse direction leads to minor energy contributions. It is therefore reasonable to expect microstructure which is only compatible in-plane \cite{ChaudhuriMueller2006,Hornung2008}. In this case, the two eigenstrains have the same determinant as $3\times 3$ matrices, but there is no reason for the $2\times 2$ blocks which determine the in-plane microstructure to have the same determinant,
as for example with  the three-dimensional eigenstrains
\begin{equation*}
 A_2'= 
 \begin{pmatrix}
   1 & 0 & 0\\
 0 & 1-\alpha & 0 \\
0 & 0 &1/(1-\alpha)
 \end{pmatrix} \text{ and }
 \quad B_2'=
 \begin{pmatrix}
   1 & 0 & 0\\
 0 & 1+\alpha & 0 \\
0 & 0& 1/(1+\alpha)
 \end{pmatrix}\,.
\end{equation*}
It would be interesting to see if the difference between $K_1$ and $K_2$ results in a different scaling in the thin-film geometry studied in  \cite{ChaudhuriMueller2006,Hornung2008}. 

If no rank-one connection is
present, then 
the  energy is much larger, and typically no microstructure is
formed. 
 
Before stating the main results we introduce the key parameters and the set of
admissibile functions. We focus on a rectangle, and assume throughout
\begin{equation}\label{assumption:1}
\Omega= (0,L)\times (0,H), \quad H,L>0, \quad \alpha\in(0,1),\quad \varepsilon>0\,,
\end{equation}
in some parts additionally $\alpha< 1/2$.
We shall impose Dirichlet boundary conditions corresponding to the  average of
$A_j$ and $B_j$, which, with the present choices, is in both cases 
the  identity, $\frac12(A_j+B_j)=\Id$. Precisely, the
 admissible deformations are in the set
\begin{equation}\label{assumption:2}
\calM=\{u\in W^{1,2}(\Omega;\R^2): \ Du\in BV(\Omega;\R^{2\times 2}), \ u(x)=x \ \forall x\in \partial\Omega \}.
\end{equation}
For $j\in\{1,2\}$, we are interested in the scaling of the minimum of
\begin{equation}\label{assumption:E}
E_j^\varepsilon[u,\Omega] = \int_{\Omega}\dist^2(Du,K_j)dx + \varepsilon
|D^2u|(\Omega)
\end{equation}
over all $u\in\calM$. Existence of minimizers follows immediately from the
direct method of the calculus of variations and will not be further addressed
here. 
In the following, $c$ will denote a generic positive constant, which may change 
its value from line to line but in particular does not depend on 
the four parameters $\alpha$, $\varepsilon$, $H$ and $L$. We shall denote by
$E_j^\eps[v,\omega]$  the integral of the  expression in (\ref{assumption:E})
 over a subset $\omega\subset \Omega$.

\begin{theorem}\label{Theorem:K1}
There is $c>0$ such that, 
under the assumptions
(\ref{assumption:K1}), (\ref{assumption:1}), (\ref{assumption:2}) and
 (\ref{assumption:E}), one has
\begin{equation*}
\frac{1}{c}f(\alpha,\varepsilon,L,H) \leq \min_{u\in\calM} E_1^\eps[u,\Omega] \leq c
f(\alpha,\varepsilon,L,H)\,,
  \end{equation*}
where
  \begin{alignat}1\nonumber
    f(\alpha,\varepsilon,L,H)=\min&\Bigl\{
\alpha^{4/3}\varepsilon^{2/3}L^{1/3}H+\alpha\varepsilon L,\\
\nonumber
&\phantom{\Bigl\{}
\alpha^{4/3}\varepsilon^{2/3}LH^{1/3}+\alpha^4LH+\alpha\varepsilon H,\\
&\phantom{\Bigl\{}
\alpha^2 LH\Bigr\}\,.
\label{eqdefFtheo1}
  \end{alignat}
The constant $c$ does not depend on $\alpha$, $\varepsilon$, $L$
and $H$. 
\end{theorem}
\begin{proof}
  Follows from  Lemma \ref{lemmaUBK1} and Lemma \ref{lemmaLBK1}.
\end{proof}
\begin{figure}[t]
  \centering
  \includegraphics[width=11cm]{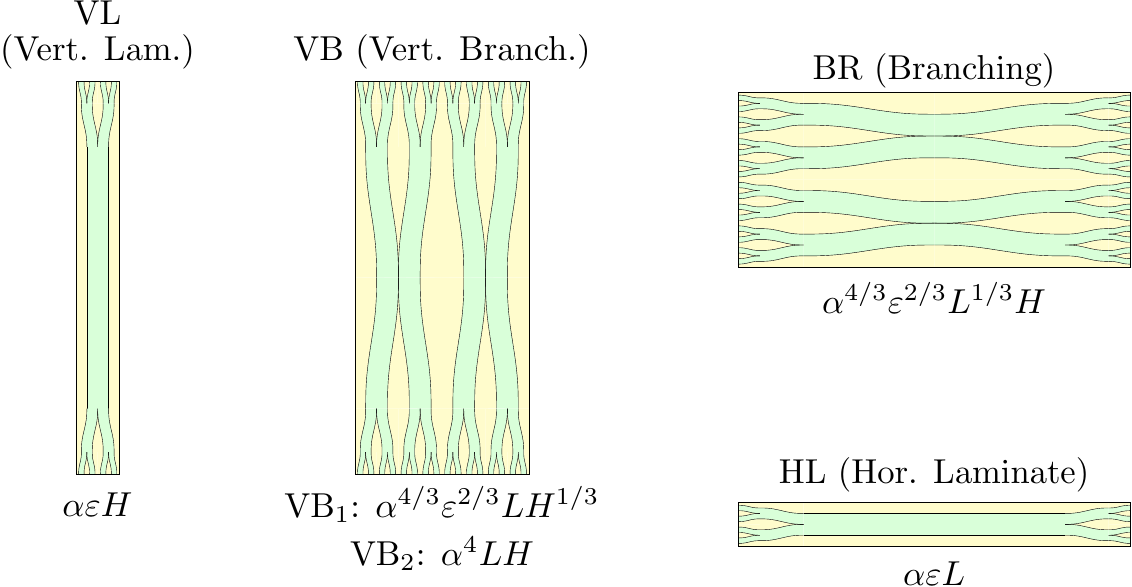}
   \caption{Sketches of the deformation $u$ in the different phases 
appearing in Theorem \ref{Theorem:K1}.}
  \label{fig:PD2}
\end{figure}
Theorem \ref{Theorem:K1} shows that, depending on the values of the four
parameters, different types of microstructure give the optimal energy
scaling. In 
Figure \ref{fig:PD2} we sketch the different microstructures
(using the constructions from the upper bound) and in 
Figure \ref{fig:PD1} the corresponding phase diagram.
 We now briefly discuss the different phases.
\begin{figure}[t]
  \centering
  \includegraphics[width=8cm]{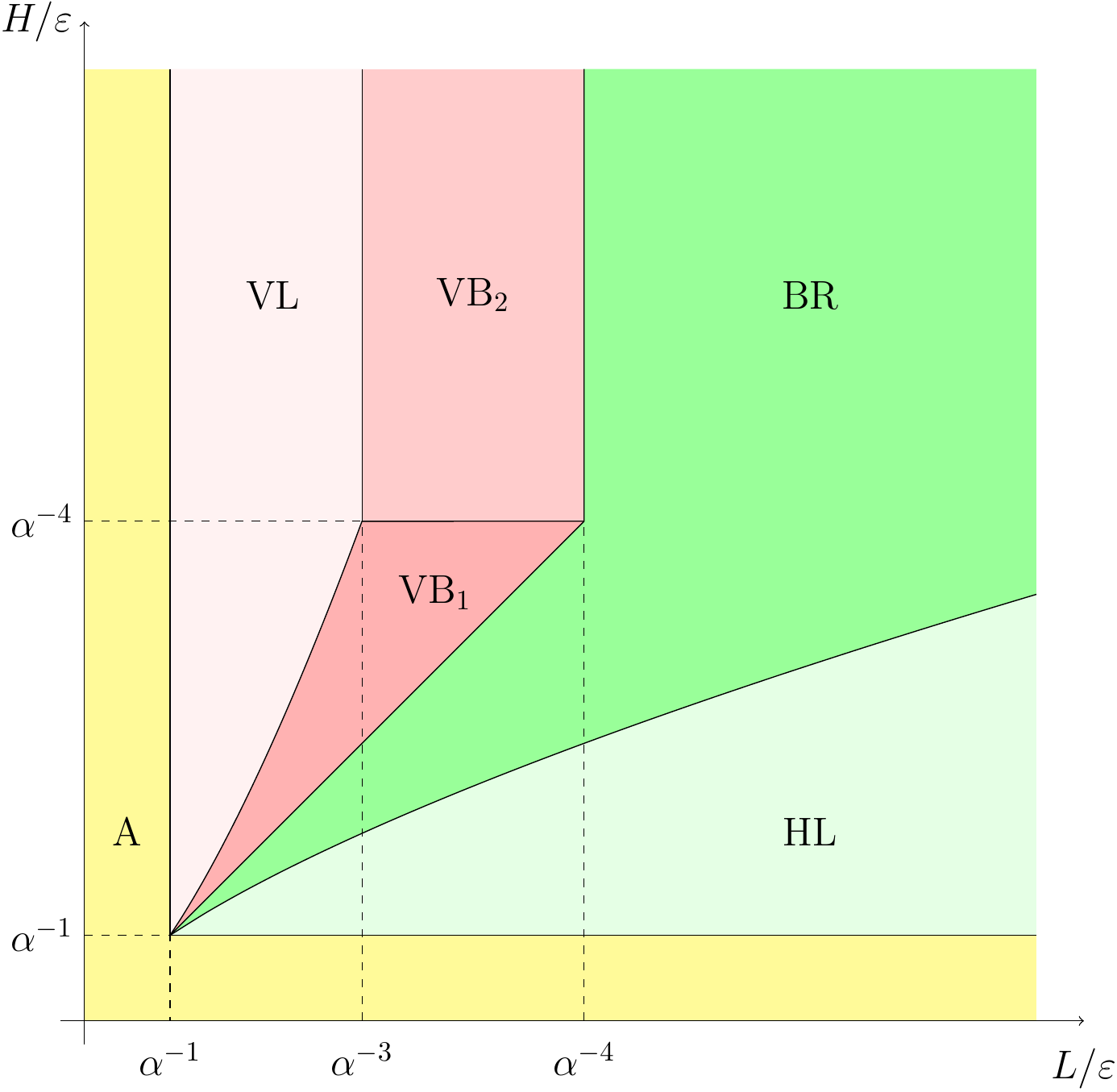}
  \caption{Schematic phase diagram in case $K_1$, in the $L/\eps$, $H/\eps$
    plane.}
  \label{fig:PD1}
\end{figure}
\begin{enumerate}
\item[A] If the domain is very small,
then the boundary data and the regularization term dominate the picture
and the transformation to martensite is not energetically convenient. The energy
scales as 
$\alpha^2 LH$, which is easily attained by the deformation $u(x)=x$. We call
this austenitic (A) phase. From (\ref{eqdefFtheo1}), one easily sees that this
phase is the optimal one if $L/\eps<\alpha^{-1}$ or $H/\eps<\alpha^{-1}$.
\item[BR]
In the branching regime (BR) the deformation gradient $Du$ oscillates  between
a value close to $A_1$ 
and a value close to $B_1$ in the interior of the sample; the boundaries are approximately horizontal
and the oscillations
become finer approaching the left and right boundary. The total energy, if $L$
and $H$ are sufficiently large compared to $\eps$, is proportional to
$\alpha^{4/3}\eps^{2/3}L^{1/3}H$, as predicted by Kohn and Müller. 
\item[HL] Horizontal laminate. If the sample is long and thin, then it is convenient to have a
  single oscillation in most of the sample, which then branches only in small
  regions close to the left and right boundary. The dominant energy contribution
 then originates from the single long interface and is  proportional to $\alpha\eps 
  L$. 
\item[VB] Vertical branching. If $H$ is larger than $L$, then it may be more
  convenient to use the second rank-one connection and have a vertical
  branching, qualitatively similar to a 90-degrees rotation of regime BR. The
  rotation, since we are dealing with a fully nonlinear model, 
  does not bring the matrices $A_1$, $B_1$ to the matrices  $\Id\pm \alpha
  e_1\otimes e_2$, but instead   leaves a second-order perturbation, which
  generates an additional energy   cost proportional to $\alpha^4LH$ (see discussion in Lemma \ref{lemmaUBK1}
  for details). Therefore, this regime can be subdivided into two parts: for
  small $H$, the $\alpha^{4/3}\eps^{2/3}LH^{1/3}$ term dominates (regime
  VB$_1$), for large $H$, the $\alpha^4LH$ term dominates (regime VB$_2$). 
\item[VL] As in the case of phase HL, if $L$ is very small, then a long
  vertical interface will dominate.
\end{enumerate}

The boundaries between the different regions can be easily obtained from 
(\ref{eqdefFtheo1}), and are sketched in Figure \ref{fig:PD1}.

\begin{figure}[t]
  \centering
  \includegraphics[width=11cm]{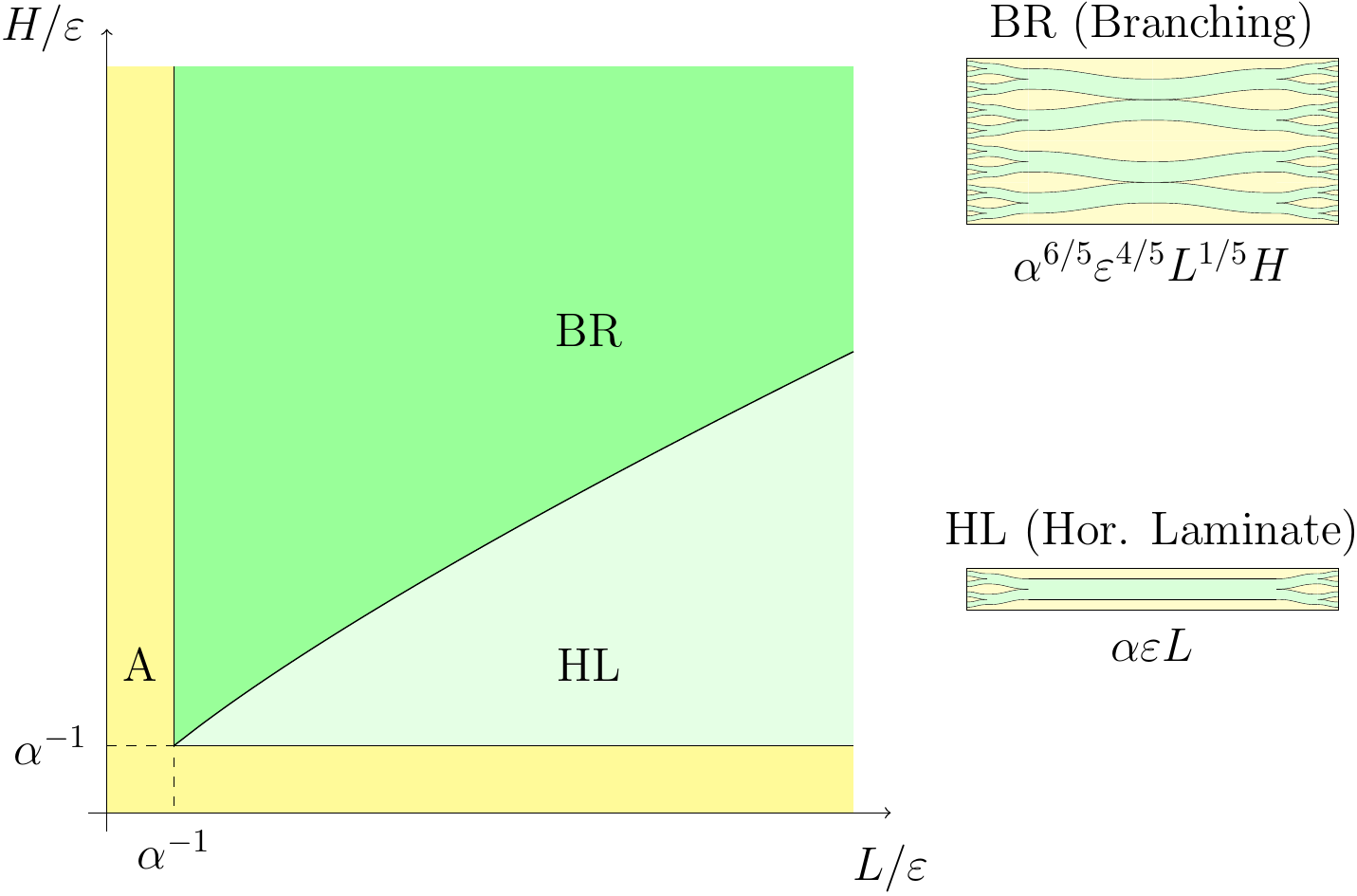}
  \caption{Schematic phase diagram in case $K_2$, in the $L/\eps$, $H/\eps$
    plane.}
  \label{fig:PD3}
\end{figure}

In the case of a single rank-one connection the energy scaling is different,
and the phase diagram simpler, see Figure \ref{fig:PD3}. Here, only the three regimes A, HL and BR play a role.
\begin{theorem}\label{Theorem:K2}
There is $c>0$ such that, 
under the assumptions
(\ref{assumption:K2}), (\ref{assumption:1}), (\ref{assumption:2}) and
 (\ref{assumption:E}), and for $\alpha<1/2$,   one has
 \begin{equation*}
 \frac{1}{c}g(\alpha,\eps,L,H) \leq \min_{u\in\calM} E_2^\eps[u,\Omega]  \leq c g(\alpha,\eps,L,H) \,,
 \end{equation*}
where
\begin{equation*}
g(\alpha,\eps,L,H)=\min\Bigl\{\alpha^{6/5}\varepsilon^{4/5}L^{1/5}H+\alpha\eps L
,\alpha^2 LH\Bigr\}\,.
\end{equation*}
The constant $c$ does not
depend on $\alpha$, $\varepsilon$, $H$ 
and $L$. 
\end{theorem}
\begin{proof}
  Follows from  Lemma \ref{lemmaUBK2} and Lemma \ref{lemmaLBK2}.
\end{proof}

The situation in which the dependence on $\alpha$ is not analyzed is
substantially simpler, in particular for the lower bounds. Indeed, a result
similar to Theorem \ref{Theorem:K1} without 
dependence on $\alpha$ and with $H=L$ was first proven in
\cite{B9_Chan,B9_ContiOWJun2007}.
These results  are  discussed in  the PhD thesis
\cite{B9_ChanDiss} and, in the simplified case $\alpha\sim 1$ and 
 $H=L$, are presented in
\cite{ChanContiSFBBericht}. 
Here we focus for simplicity on the case that $\Omega$ is a rectangle, keeping the two lengths separated to emphasize, in the case of Theorem \ref{Theorem:K1}, the competition between the two orientations of the microstructure. Although the detailed comparison between the different patterns is geometry-dependent, the
leading-order scaling in $\eps$  in the  small-$\eps$ limit is not expected to depend on the shape of $\Omega$. This is addressed (for similar models) 
in \cite{B9_KohnMuller94} for parallelograms and in \cite{B9_Diermeier}
for general polygons. The case of piecewise smooth domains is treated \cite{B9_BCDM2002} for a model of thin-film blistering,
without obtaining significant differences with the case of a rectangle.

In closing this introduction, we briefly review the literature on energy scaling in martensitic phase transformations. The scalar model
\begin{equation*}
  E_\mathrm{KM}[u]=
  \begin{cases}\displaystyle
    \int_{\Omega} (\partial_1 u)^2 dx + \eps |\partial_2\partial_2 u|(\Omega) & \text{ if } \partial_2 u \in\{-1,1\}\text{ a.e.,}\\
\infty & \text{ otherwise,}
  \end{cases}
\end{equation*}
corresponding to a rigid version of our $K_1$ case, was proposed by Kohn and Müller in 1992. 
They have shown that, if the Dirichlet boundary data $u(0,\cdot)=0$ are imposed, for small $\eps$, the energy scales as $\eps^{2/3}L^{1/3}H$ \cite{B9_KohnMuller92,B9_KohnMuller94}, the same result as in Theorem \ref{Theorem:K1}. 
It was later shown that the minimizers of their functional are asymptotically self similar \cite{B9_Conti00}.
Kohn and Müller  also studied a softer problem, in which the Dirichlet boundary condition is eliminated and instead the term $\beta \|u(0,\cdot)\|_{H^{1/2}((0,H))}^2$ is added to the energy $E_\mathrm{KM}$, where $u(0,\cdot)$ is the  trace of the deformation on the $x_1=0$ boundary. Physically, this corresponds to the elastic energy of the material ``outside the domain''. Their results, which were completed in \cite{B9_Conti2006,B9_Zwicknagl},  state that, for small $\eps$ and $\beta$, the optimal energy scales as 
\begin{equation}\label{eqscalingKMbeta}
  \min\{\eps^{2/3}L^{1/3}H, \eps^{1/2}\beta^{1/2}L^{1/2}H\}\,.
\end{equation}
The first option corresponds to a branched pattern, the second one to a laminar, one-dimensional pattern, in which the interfaces are exactly parallel to $e_1$ and the incompatibility is accommodated by the boundary term.
We remark that the minimum in the assertion of
 Theorem 1 in \cite{B9_Conti2006} misses a term proportional 
to $\eps^{1/2}$. This term,
which  was correctly identified in 
 \cite{B9_Zwicknagl}, arises since,
at the end of the proof, one can only choose $l'=\min\{1/\beta, \eps^{-1/2}\}$
and is only relevant for large values of $\eps$.
The correct bound is
$J_{\eps,\beta}\ge c \min\{\eps^{1/2}\beta^{1/2}L^{1/2}H,
\eps^{2/3}L^{1/3}H, \eps^{1/2}H^{3/2}, \beta H^2\}$.

Variants, in which the sharp condition $\partial_2 u\in\{-1,1\}$ is replaced by a penalization and using different norms for the singular perturbation, were studied by Zwicknagl \cite{B9_Zwicknagl}, leading   for small $\eps$ to 
a result similar to (\ref{eqscalingKMbeta}).
The geometrically nonlinear version was addressed by Dolzmann and Müller
\cite{B9_DolzmannMueller1995}, who proved a rigidity result that implies 
that the energy scaling of the present problem is superlinear in $\eps$. 
Lorent \cite{Lorent2006} has shown that the energy scaling of the singularly perturbed problem and the one of a finite-element approximation are the same,
$\Gamma$-convergence to a sharp-interface model (after dividing the 
energy by $\eps$)  was derived in \cite{B9_ContiSchweizer2006b},
quantitative rigidity estimates for low-energy states of  singularly
perturbed multiwell problems
were obtained in \cite{B9_Lorent2005,B9_ContiSchweizer2006b,B9_ChermisiConti2010,JerrardLorent2013}.
The energy scaling for the geometrically nonlinear model was obtained in \cite{B9_Chan,B9_ContiOWJun2007}, a version of Theorem
\ref{Theorem:K1} with a geometrically linear model and without the explicit
dependence on $H$ and $L$ was given in \cite{B9_Diermeier}.
A somewhat different approach, in which not only the deformation but also the shape of the inclusion was optimized, 
was followed in \cite{KnuepferKohnOtto2013} in a geometrically linear setting, mimicking nucleation of martensite and leading to different scalings.

The rest of this paper discusses the proofs of Theorem \ref{Theorem:K1} and Theorem \ref{Theorem:K2}. The upper bounds are proven in Section \ref{sectUB}, the 
 lower bounds  in Section \ref{seclowerbound}.

\section{Upper bounds}\label{sectUB}
The upper bounds are proven by explicit constructions. Whereas the basic
construction in the proof of Theorem \ref{Theorem:K1} is a direct
generalization of the one by Kohn and Müller, the different scaling in Theorem
\ref{Theorem:K2} arises from a somewhat different  construction, which 
exploits the vectorial nature of the problem and the invariance under
rotations  of the energy density. Therefore, we start with a discussion of the latter, and postpone
the proof of  Theorem \ref{Theorem:K1} to Section \ref{secubtheo1} below.
In what follows,
$c$  denotes a positive constant which does not depend on $\alpha$,
$\varepsilon$, $L$  and $H$  and may be different from line to line.  In this
section we write $(x,y)$ for generic points in $\R^2$.

\subsection{Proof of the upper bound in Theorem \ref{Theorem:K2}}
\label{secUBtheo2}
The key idea is to use a number of period-doubling steps which join fine-scale
oscillations close to the boundary with coarse-scale oscillations in the
interior of the sample. The difference between the two cases considered and
the main novelty in the present construction reside in the treatment of the
period-doubling step, which we now present.

In order to fix the boundary conditions on the internal boundaries,
we fix a continuous $1$-periodic function $\vartheta:\R\to\R$ such that $\vartheta(0)=0$, 
$\vartheta'=1$ on $(-1/4,1/4)$, $\vartheta'=-1$ on $(1/4,3/4)$
(equivalently, $\vartheta(t)=\dist(t+\frac14,\Z)-\frac14$)
and scale it according  to $\vartheta_h(t)=h\vartheta(t/h)$. Then $\vartheta_h$ is $h$-periodic
and has derivative $\pm1$ almost everywhere. We observe that the function
\begin{equation*}
  w_h(x,y)=
  \begin{pmatrix}
    x\\y +\alpha\vartheta_h(y)
  \end{pmatrix}
\end{equation*}
fulfills $Dw_h\in\{A_2,B_2\}$ almost everywhere.
Further, on each line
$\{(x,y): y\in h\Z\}$ we have $w_h(x,y)=(x,y)$, corresponding to
the prescribed boundary conditions. The construction will be based on using $w_h(x,y)$
for selected values of $x$, with different values of $h$, larger in the center
of the domain and smaller close to $x=0$ and $x=L$. The key ingredient in the
proof is the following period-doubling construction, which permits to join different values of
$h$ with small energetic cost.
\begin{lemma}\label{eqconstronerecttheo2}
  Let $0<h\le\ell$, $\omega=(x_0,x_0+\ell)\times(y_0,y_0+h)$, $\alpha\in(0,1)$, 
  $\eps>0$. Then there is a function 
  $v:\omega\to\R^2$ such that
  \begin{equation*}
    E_2^\eps[v, \omega]\le 
    c \left(\alpha^2\frac{h^5}{\ell^3} +  \alpha \eps \ell\right)\,,
  \end{equation*}
with $\|Dv-\Id\|_{L^\infty}\le c\alpha$ and fulfilling
 the boundary conditions $v(x,y)=(x,y)$ for $y\in\{y_0,y_0+h\}$,
$v(x_0,y)=(x_0,y+\alpha \vartheta_{h/2}(y-y_0))$, 
$v(x_0+\ell,y)=(x_0+\ell,y+\alpha\vartheta_{h}(y-y_0))$.
The constant $c$ is universal.
\end{lemma}
\begin{proof}
  We assume without loss of generality $x_0=y_0=0$. 
  We fix a smooth interpolation function  $\gamma:[0,1]\rightarrow [0,1]$ with
$\gamma(0)=0$, $\gamma(1)=1$ and the first two derivatives vanishing at both
points, say  $\gamma(x)=10x^3-15x^4+6x^5$. We subdivide  $\omega$ into the five sets
\begin{equation*}
 \begin{split}
  \omega_1 &= \left\{ (x,y)\in\omega: \ 0< y <
    \frac{h}{8}+\frac{h}{8}\gamma\left(\frac x\ell\right)\right\}, \\
  \omega_2 &= \left\{ (x,y)\in\omega: \
    \frac{h}{8}+\frac{h}{8}\gamma\left(\frac
      x\ell\right)<y<\frac{3h}{8}+\frac{h}{8}\gamma\left(\frac x\ell\right)\right\}, \\ 
  \omega_3 &= \left\{ (x,y)\in\omega: \
    \frac{3h}{8}+\frac{h}{8}\gamma\left(\frac
      x\ell\right)<y<\frac{5h}{8}-\frac{h}{8}\gamma\left(\frac x\ell\right)\right\}, \\ 
  \omega_4 &= \left\{ (x,y)\in\omega:  \ \frac{5h}{8}-\frac{h}{8}\gamma\left(\frac x\ell\right)<y<\frac{7h}{8}-\frac{h}{8}\gamma\left(\frac x\ell\right)\right\}, \\
  \omega_5 &= \left\{ (x,y)\in\omega: \ \frac{7h}{8}-\frac{h}{8}\gamma\left(\frac x\ell\right)<y<h\right\},
 \end{split}
\end{equation*}
see Figure \ref{fig:2}. In the domains $\omega_1$,  $\omega_3$ and $\omega_5$
we shall construct $v$ so that $Dv$ is close to $B_2=\diag(1,1+\alpha)$, in the other two so that
it is close to $A_2=\diag(1,1-\alpha)$. Since the boundaries are curved, there
will necessarily be deviations from $A_2$ and $B_2$, and in particular 
the $x$-derivative will not be identically $e_1$.

\begin{figure}[t]
  \centering
  \includegraphics[width=10cm]{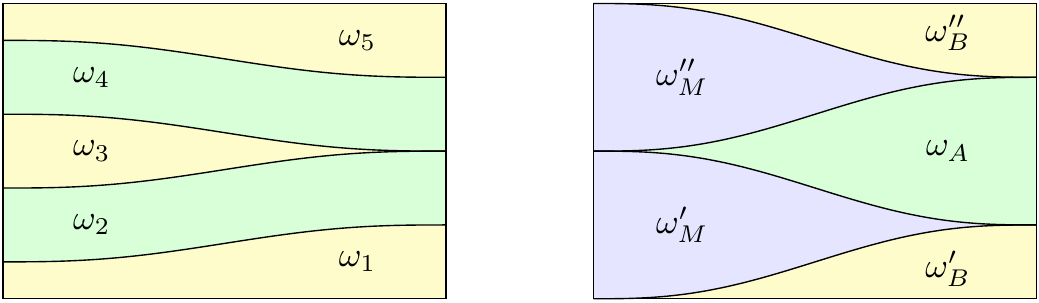}
  \caption{Left panel: subdivision of the domain $\omega$ used in Lemma
    \ref{eqconstronerecttheo2}. Right
    panel: subdivision of the domain used in Lemma
    \ref{eqconstronerecttheo2bl}.}
  \label{fig:2}
\end{figure}

The basic idea is to  set $v_2(\cdot,0)=0$, according to the boundary conditions,
and then to construct $v_2(x,y)$
by integrating $\partial_2 v_2(x,\cdot)$ from $0$ to $y$, with $\partial_2
v_2=1+\alpha$ in $\omega_1$,  $\omega_3$,  $\omega_5$,  and $\partial_2
v_2=1-\alpha$ on the other two sets. This leads to
\begin{equation*}
v_2(x,y)=
\begin{cases}
{(1+\alpha)y} & \text{on $\omega_1$}, \\
{(1-\alpha)y+\alpha h(\frac{1}{4}+\frac{1}{4}\gamma\left(\frac x\ell\right))} &
\text{on $\omega_2$}, \\ 
{(1+\alpha)y-\frac{\alpha h}{2}} & \text{on $\omega_3$}, \\
{(1-\alpha)y+\alpha h (\frac{3}{4}-\frac{1}{4}\gamma\left(\frac x\ell\right))} & \text{on $\omega_4$}, \\
{(1+\alpha)y-\alpha h} & \text{on $\omega_5$}.
\end{cases}  
\end{equation*}
Computing $\partial_1 v_2$, we easily see that it is of order $\alpha
h/\ell$ in $\omega_2$ and $\omega_4$. In particular, if we would simply set
$v_1(x,y)=x$, we would obtain
\begin{equation}\label{eqdval}
  |Dv-A_2|\sim \frac{\alpha h}{\ell} \text{  in  $\omega_2$ and $\omega_4$}\,,
\end{equation}
and therefore an energy at least of order $\alpha^2h^3/\ell$ in $\omega$.

It is, however, possible, via a more careful construction of $v_1$, to  reduce the
energy further. The idea is to use rotations, in a way similar to what was done in the
study of folding pattern in blistered thin films
\cite{B9_BCDM2000,B9_BCDM2002}. To understand the key idea it is helpful to
first consider  the 
linearized model, in which $\dist(Dv, \SO(2) A_2)$ is replaced by
$|(Dv+Dv^T)/2-A_2|$. The easier choice for $v_1$, namely, $v_1(x,y)=x$, would
make the first diagonal entry in the gradient, $\partial_1v_1$, vanish. The other natural option, namely, to
make the off-diagonal entries (which are equal) vanish, leads to the
possibility of defining  $v_1$ by solving
$\partial_2 v_1+\partial_1v_2=0$. 
By the same estimate as in (\ref{eqdval}), the term $\partial_1v_2$ is of
order $\alpha h/\ell$, hence integrating in $y$ we  would obtain a component $v_1$ of order $\alpha
h^2/\ell$. 
The diagonal entry in the gradient
 $\partial_1 v_1$  would then be of
 order $\alpha (h/\ell)^2$. The additional factor $h/\ell$ in $\partial_1v_1$
 with respect to $\partial_2v_1$ arises from the
 fact that we
took an integral in the $y$ direction and a derivative in the
$x$-direction. On most of our domain, $h$ will be substantially smaller 
than $\ell$. Therefore, this construction leads to a substantial reduction in energy. 

Since we are dealing with a geometrically nonlinear
model, there will be  an additional error term
of order
$(\partial_1 v_2)^2\sim \alpha^2 (h/\ell)^2$,
 arising from the linearization. We also need to take care of the
fact that we are not linearizing around the identity, but instead around the
matrix $A_2$. Therefore, we solve 
$\partial_2 v_1+(1-\alpha)\partial_1v_2=0$ instead of $\partial_2
v_1+\partial_1v_2=0$, see (\ref{eqlinearize}) below.

A detailed computation based on this idea  motivates the definition 
\begin{equation*}
v_1(x,y)=
\begin{cases}
{x} & \text{on $\omega_1$}, \\
{x+\alpha(1-\alpha)\frac{h}{4\ell}\gamma'\left(\frac
    x\ell\right)(\frac{h}{8}+\frac{h}{8}\gamma\left(\frac x\ell\right)-y)} & \text{on $\omega_2$},
\\ 
{x-\alpha(1-\alpha)\gamma'\left(\frac x\ell\right)\frac{h^2}{16\ell}}& \text{on $\omega_3$}, \\
{x-\alpha(1-\alpha)\frac{h}{4\ell}\gamma'\left(\frac
    x\ell\right)(\frac{7h}{8}-\frac{h}{8}\gamma\left(\frac x\ell\right)-y)} &
\text{on $\omega_4$}, \\ 
{x} & \text{on $\omega_5$}.
\end{cases}  
\end{equation*}
This concludes the definition of $v$. We remark that $v_1$ satisfies the boundary conditions at $x=0$ and $x=\ell$ in the regions $\omega_2$, 
 $\omega_3$ and
 $\omega_4$ only if $\gamma'(0)=\gamma'(1)=0$.  Therefore, in this case (at variance with the $K_1$ case, see below) the simple choice  $\gamma(s)=s$ would not have been possible.
The boundary conditions on $\gamma''$, which give continuity of $Dv$ across vertical boundaries, are only a matter of convenience.

In the rest of the proof, we estimate the energy of the function we constructed.
First, we easily check that $v$ is continuous and obeys the boundary conditions.
By construction, $Dv=B_2\in K_2$ in the sets $\omega_1$ and $\omega_5$. A short
computation shows that $|Dv-B_2|=|\partial_1v_1-1|\le c \alpha h^2/\ell^2$ in $\omega_3$.
In order to estimate the energy in $\omega_2$ and $\omega_4$,
we compute
\begin{equation}\label{eqlinearize}
  \begin{split}
 \min_{Q\in \SO(2)} |Dv-QA_2|
\le& \min_{\varphi\in\mathbb{R}}
|\partial_1v_1-\cos\varphi|+|\partial_2v_2-(1-\alpha)\cos\varphi|\\
&
+|\partial_1v_2-\sin\varphi|+|\partial_2v_1+(1-\alpha)\sin\varphi|\\
  \le& |\partial_1 v_1-1|+|\partial_2 v_2-(1-\alpha)| \\
&+ 
|\partial_2 v_1+(1-\alpha)\partial_1  v_2| + c |\partial_1 v_2|^2    \,,
  \end{split}
\end{equation}
where we have chosen $\varphi$ such that $\sin\varphi=\partial_1v_2$ if
$\partial_1v_2\in[-1,1]$, $\varphi=0$ otherwise. 
Inserting the definition of $v$, one sees that the second and third term
vanish,  the remainder can be estimated by
\begin{alignat*}1
\mathrm{dist}(Dv,K_2)\le& 
c\alpha \frac{h^2}{\ell^2} \|\gamma''\|_{L^\infty}
+ c\alpha\frac{h^2}{\ell^2} \|\gamma'\|_{L^\infty}^2
+ c \Bigl(\alpha \frac{h}{\ell} \|\gamma'\|_{L^\infty}\Bigr)^2\\
\le&
c\alpha \frac{h^2}{\ell^2}\quad \text{ in }\omega_2\cup\omega_4\,,
\end{alignat*}
and therefore,
\begin{equation*}
 \int_{\omega_2\cup\omega_4}\dist^2(Dv(x,y),K_2)\, dxdy  \leq
 c\alpha^2\frac{h^5}{\ell^3}. 
\end{equation*}
A simple computation shows that $|Dv-\Id|\le c\alpha$ on $\omega$.

We finally turn to the interfacial part of the energy. 
In each of the $\omega_i$ we see by inspection that $|D^2v|\le c \alpha
h/\ell^2$ (we use repeatedly that $h\le\ell$ to simplify expressions). The jump
contributions can be estimated by $2\|Dv-\Id\|_{L^\infty}$ 
times the length of the jump set; since $\gamma$ is monotone, this is bounded
by $4(h+\ell)$. We conclude that
\begin{equation*}
\varepsilon |D^2v|(\omega)\leq c\varepsilon\left(\alpha (\ell+h) + \alpha
  h\ell \frac{h}{\ell^2}\right)\leq c\varepsilon \alpha \ell\,.
\end{equation*}
This completes the proof.
\end{proof}

We now turn to the boundary layer. 
One could use a linear interpolation as in \cite{ChanContiSFBBericht}. For 
the sake of variety, we
present an alternative construction, which is based on a variant of the
construction of  Lemma \ref{eqconstronerecttheo2} and  does not use gradients far
from $A$, $B$ and $\Id$.
\begin{lemma}\label{eqconstronerecttheo2bl}
  Let $0<h\le\ell$, $\omega=(x_0,x_0+\ell)\times(y_0,y_0+h)$, $\alpha\in(0,1)$, $\eps>0$.
 Then there is a function    $v:\omega\to\R^2$ such that
  \begin{equation*}
    E_2^\eps[v, \omega]\le 
 c \left(\alpha^2h\ell +  \alpha \eps \ell\right)\,,
  \end{equation*}
  with the boundary conditions $v(x,y)=(x,y)$ for $y\in\{y_0,y_0+h\}$ or $x=x_0$,
$v(x_0+\ell,y)=(x_0+\ell,y+\alpha\vartheta_{h}(y-y_0))$ and
$\|Dv-\Id\|_{L^\infty}\le c\alpha$.
The constant $c$ is universal.
\end{lemma}
\begin{proof}
  The construction is similar to the one in Lemma \ref{eqconstronerecttheo2}.
  Again, it suffices to treat the case $x_0=y_0=0$.
With   $\gamma$  as in  Lemma \ref{eqconstronerecttheo2}, we set
\[
\begin{split}
 \omega_B' &= \left\{ x\in (0,\ell), \ 0<y<\frac{h}{4}\gamma
   \left(\frac{x}{\ell}\right)\right\}, \\
 \omega_M' &= \left\{ x\in(0,\ell), \
   \frac{h}{4}\gamma\left(\frac{x}{\ell}\right)<y<\frac{h}{2}-\frac{h}{4}\gamma\left(\frac{x}{\ell}\right)
\right\}, 
 \\ 
 \omega_A &= \left\{x\in (0,\ell), \
   \frac{h}{2}-\frac{h}{4}\gamma\left(\frac{x}{\ell}\right)<y<
   \frac{h}{2}+\frac{h}{4}\gamma\left(\frac{x}{\ell}\right)
\right\}\,,\\
 \omega_M'' &= \left\{ x\in(0,\ell), \
\frac{h}2+\frac{h}{4}\gamma\left(\frac{x}{\ell}\right)<y<h-\frac{h}{4}\gamma\left(\frac{x}{\ell}\right)
\right\}, 
 \\ 
 \omega_B'' &= \left\{ x\in (0,\ell), \ h-\frac{h}{4}\gamma \left(\frac{x}{\ell}\right)<y<h\right\}, 
 \end{split}
\]
see Figure \ref{fig:2}.
We set $v_1(x,y)=x$ and
\begin{equation*}
  v_2(x,y)=
 \begin{cases}
 y+\alpha y & \text{on $\omega_B'$}, \\
 y+\frac{1}{4} \alpha h\gamma\left(\frac{x}{\ell}\right) & \text{on
   $\omega_M'$}, \\ 
 y+\alpha(\frac{h}{2}-y) & \text{on $\omega_A$},\\
 y-\frac{1}{4}\alpha h\gamma\left(\frac{x}{\ell}\right) & \text{on
   $\omega_M''$}, \\ 
 y+\alpha (y-h) & \text{on $\omega_B''$}.
 \end{cases}
\end{equation*}
One then checks that $v_2$ is continuous, $|D^2v_2|\le c\alpha h/\ell^2$
inside each of the five sets, $Dv=B_2$ on  $\omega_B'\cup\omega_B''$, $Dv=A_2$ in
$\omega_A$, and $|Dv-\Id|\le 
\alpha h \|\gamma'\|_{L^\infty}/\ell$ in $\omega_M'\cup\omega_M''$. Therefore, $\|Dv-\Id\|_{L^\infty}\le c\alpha$ and, treating the jump terms as above,
\[
 E_2^{\varepsilon}[v,(0,\ell)\times (0,h)] \leq c\alpha^2
h\ell+c\eps\alpha \ell \,.
\]
\end{proof}

\begin{figure}[t]
  \centering
  \includegraphics[width=12cm]{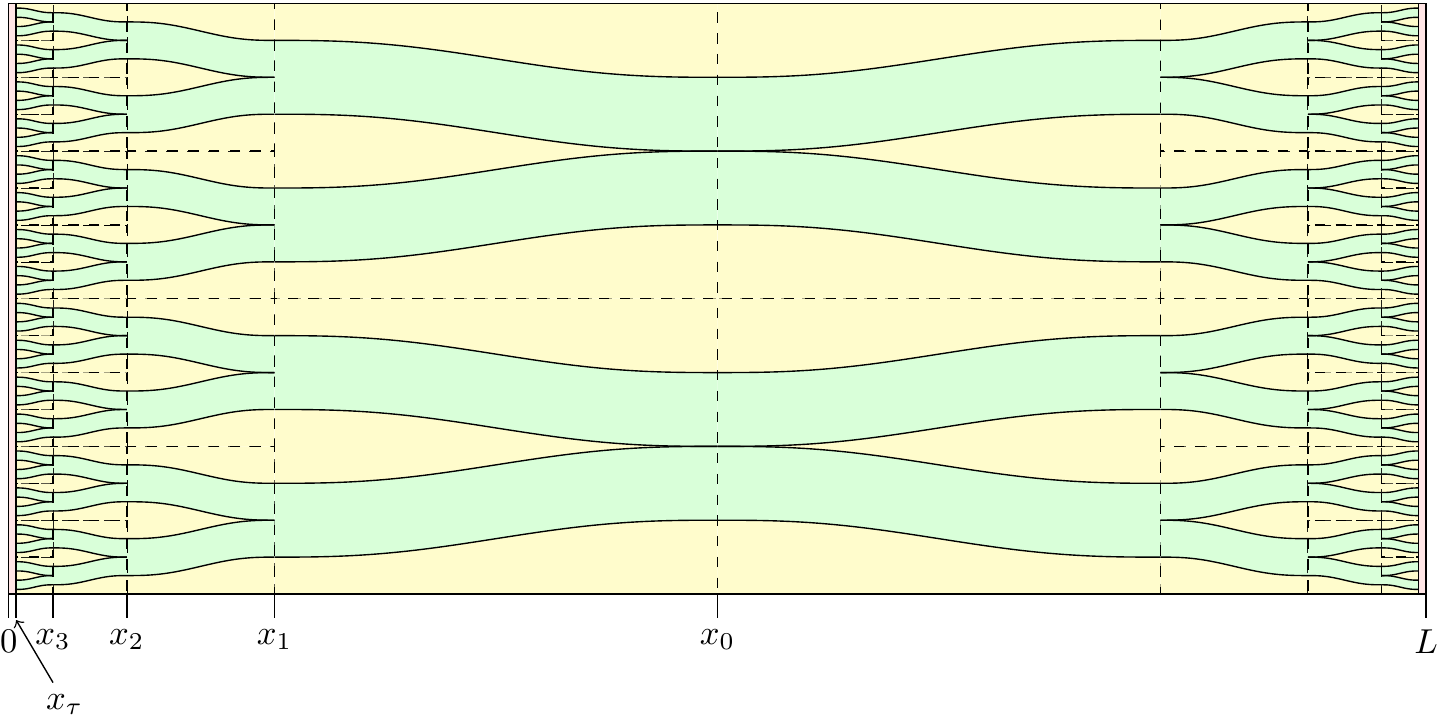}    
  \caption{Global pattern used in Lemma \ref{lemmaUBK2} for the construction of the upper bound of
    Theorem \ref{Theorem:K2}.}
  \label{fig:1}
\end{figure}

We are now ready to present the full construction.
\begin{lemma}\label{lemmaUBK2}
Under the assumptions of  Theorem \ref{Theorem:K2}, there is a deformation
$u\in \calM$ such that
\begin{equation*}
  E_2^\eps[u,\Omega]\le  c \min\Bigl\{ \alpha^{6/5}\eps^{4/5} L^{1/5}H + \alpha
  \eps L, \alpha^2 LH\Bigr\}\,.
\end{equation*}
The constant $c$ is universal.
\end{lemma}
\begin{proof}
Inserting $u(x)=x$, we easily obtain a construction with energy $\alpha^2
LH$. Therefore, we only need to consider the case in which the first term in
the minimum is smaller than the second. Then, necessarily,
$\eps^{4/5}\alpha^{6/5} L^{1/5}H \le \alpha^2 LH$, which is equivalent to
\begin{equation}\label{eqepsalphaL}
  \eps \le \alpha L\,.
\end{equation}

We now start the construction. We will focus on $\Omega'=(0,L/2)\times (0,H)$, the other part is analogous
(with minor sign changes in the formulas).
The basic idea in the
construction is to have coarse folds in the central part, which then refine in
a self-similar way approaching to the boundary, up to a point where a
different construction for the boundary layer is needed, see Figure \ref{fig:1}.

Let $N\in\N\setminus\{0\}$ be the number of oscillations in the center, $\theta\in(1/4,1/2)$ a
geometric factor, both will be chosen below.
We cut the domain in vertical stripes, so that in each of them, say
$(x_{i+1},x_i)\times(0,H)$, the  construction will be $h_i$-periodic in the
$y$ direction. We define
\begin{equation*}
  x_i = \frac{L}{2}\theta^{i}\quad \text{ and } \quad  h_i= \frac{H}{2^{i}N}
\end{equation*}
and  denote  the width of the $i$-th stripe by
$\ell_i=\theta^{i}(1-\theta)\frac{L}{2}$. 

Since the construction in Lemma \ref{eqconstronerecttheo2} requires
$h_i \leq \ell_i$, we need to stop
this procedure at a finite $\tau$, defined as the largest integer $i$ with $h_i\leq \ell_i$.
Further, we need to choose $N$ such that $H/N\le (1-\theta)L/2$, to be able to start.

For $i\in [0,\tau]\cap \N$ and $y\in (0,H)$, we set
\begin{equation}\label{eqvperiodiclines}
  v(x_i,y)=\vii{x_i}{y+\alpha \vartheta_{h_i}(y)}\,.
\end{equation}
This permits to treat each refinement step separately. In each stripe 
of the form $(x_{i+1},x_i)\times(0,H)$, $i=0, \dots, \tau-1$, we use
$N2^i$ times the construction of Lemma \ref{eqconstronerecttheo2}.  
In the set $(0,x_\tau)\times (0,L)$ we use the same procedure with Lemma
\ref{eqconstronerecttheo2bl}. The result is continuous thanks to 
(\ref{eqvperiodiclines}).

It remains to estimate the energy. This is given by the sum over the stripes
of $N2^i$ times the energy in each rectangle plus the boundary energy, which
on the boundary of each rectangle is controlled by $\eps\|Dv-\Id\|_{L^\infty}
(h_i+\ell_i)\le c \alpha \eps \ell_i$. We obtain
\begin{equation*}
 E_2^\eps[u,\Omega'] \leq c\sum_{i=0}^{\tau-1} 2^iN \left( 
\alpha^2 \frac{h_i^5}{\ell_i^3} + \alpha \eps \ell_i\right) + c N2^\tau
\left(\alpha^2 h_\tau\ell_\tau + \alpha \eps \ell_\tau\right) \,.
\end{equation*}
By definition of $\tau$, we have $h_\tau\le\ell_\tau\le 2 h_\tau$, therefore, the
boundary term is comparable to the $i=\tau$ term in the series and can be
included in the series. 
Using 
the definitions of $h_i$ and $\ell_i$, the estimate becomes
\begin{equation*}
 E_2^\eps[u,\Omega'] \leq c\sum_{i=0}^{\tau}\left(
  \frac{\alpha^2H^5}{2^{4i}N^4\theta^{3i}L^3(1-\theta)^3} +
  \alpha\varepsilon2^{i}N\theta^{i}(1-\theta)L \right)\,.
\end{equation*}
If $\theta\in
(2^{-4/3},2^{-1})$, both $\sum_i 1/(2^4\theta^3)^{i}$ and $\sum_i (2\theta)^i$
are converging geometric series.
The precise choice of $\theta$ in this interval does not influence the scaling, but only the prefactor, which we did not compute explicitly. In principle, one could choose the value of $\theta$ that makes the energy smallest, as was discussed in \cite{B9_ChanDiss}. However, this is not the only parameter that one could minimize. For example,
it is not clear that  the optimal aspect ratios $\ell_i/h_i$ should take the form $\ell_i/h_i=\theta^i$. Further, in  converging geometric series, the initial terms play a dominant role, and it is reasonable to expect the first ones to be special. The energy of the  $i=0$  step, for example, could be lowered by using a laminate instead of a branching pattern (branching helps to reduce the energy of the ``next'' step, but for $i=0$ there is no ``next'' step). 
Since we are only interested in the scaling,  we ignore the optimization of $\theta$ and,  for definiteness,   choose an arbitrary  value. 

Setting $\theta = 2^{-5/4}$, we obtain
\begin{equation*}
 E_2^\eps[u,\Omega']\leq c \left(\frac{\alpha^2H^5}{N^4L^3} + \alpha\varepsilon NL\right) . 
\end{equation*}
We choose $N$ so that the sum is smallest,
subject to the constraints $N\ge 1$, $N\in\N$ and $N\ge 4H/L$ (as above, we only need to obtain the optimal exponents, not the prefactors). 
Precisely, we set
 $N=\lceil \alpha^{1/5}H/(\varepsilon^{1/5}L^{4/5})+4H/L\rceil$, where
$\lceil t\rceil=\min\{z\in\Z: z\ge t\}$. 
Estimating $N\ge \alpha^{1/5}H/(\varepsilon^{1/5}L^{4/5})$ in the first term
and $N\le \alpha^{1/5}H/(\varepsilon^{1/5}L^{4/5})+4H/L+1$ in the second one,
we 
conclude
\[
 E_2^{\varepsilon}[u,\Omega'] \leq c \alpha^{6/5}\varepsilon^{4/5}HL^{1/5} +c
 \alpha\eps L+ c\alpha\eps H\,.
\]
Including the set $\Omega\setminus\Omega'$ only doubles the energy. 
By (\ref{eqepsalphaL}), we see that the term $\alpha \eps H$ can be absorbed
into the first one, and the proof is concluded.
\end{proof}

\subsection{Proof of the upper bound in Theorem \ref{Theorem:K1}}
\label{secubtheo1}
The proof is similar to the one of  Theorem \ref{Theorem:K2}, but the
construction of the period-doubling step in Lemma \ref{eqconstronerecttheo2} becomes simpler. We only
discuss the differences. The presence of two rank-one connections, however, leads
in this case to the presence of two competing constructions, which are
essentially 90-degrees rotated copies of each other. The relation is discussed
in the proof of Lemma \ref{lemmaUBK1} below.

\begin{lemma}\label{eqconstronerecttheo1}
  Let $0<h\le\ell$, $\omega=(x_0,x_0+\ell)\times(y_0,y_0+h)$, $\alpha\in(0,1)$, $\eps>0$. Then there is a function 
  $v:\omega\to\R^2$ such that
  \begin{equation*}
    E_1^\eps[v, \omega]\le 
    c \left(\alpha^2\frac{h^3}{\ell} +  \alpha \eps \ell\right)\,,
  \end{equation*}
  with 
$\|Dv-\Id\|_{L^\infty}\le c\alpha$  and with
the boundary conditions $v(x,y)=(x,y)$ for $y\in\{y_0,y_0+h\}$,
$v(x_0,y)=(x_0+\alpha \vartheta_{h/2}(y-y_0),y)$, 
$v(x_0+\ell,y)=(x_0+\ell+\alpha\vartheta_{h}(y-y_0)),y)$.
The constant $c$ is universal.
\end{lemma}
\begin{proof}
Under the same assumptions as in Lemma \ref{eqconstronerecttheo2}, and with
the same domain subdivision, we define 
\[
v_1(x,y)=
\begin{cases}
 x+\alpha y & \text{ on }\omega_1, \\
 x-\alpha y +\alpha \frac{h}{4}+\alpha\frac h4\gamma(\frac x\ell) & \text{ on }\omega_2, \\
 x+\alpha y - \frac{\alpha h}{2} & \text{ on }\omega_3, \\
 x-\alpha y + \frac{3\alpha h}{4} -2 \alpha\frac h4 \gamma(\frac x\ell) & \text{ on }\omega_4, \\
 x+\alpha y -\alpha h & \text{ on }\omega_5.
\end{cases}
\]
In this case, the strain terms arising from the derivatives of $\gamma$ lie in 
the diagonal entry $\partial_1v_1$ 
 of $Dv$. Therefore their energy contribution cannot
be reduced using the other component and the rotations, hence we can
simply take $v_2(x,y)=y$.
The estimate (\ref{eqdval}) is optimal, and correspondingly we
obtain $\dist(Dv,K)\le c\alpha h/\ell$ everywhere. The strain gradient term
can be estimated in the same way  as in Lemma \ref{eqconstronerecttheo2}.
\end{proof}

The boundary layer is also similar.
\begin{lemma}\label{eqconstronerecttheo1bl}
  Let $0<h\le\ell$, $\omega=(x_0,x_0+\ell)\times(y_0,y_0+h)$, $\alpha\in(0,1)$, $\eps>0$. Then there is a function 
  $v:\omega\to\R^2$ such that
  \begin{equation*}
    E_1^\eps[v, \omega]\le 
    c \left(\alpha^2h\ell +  \alpha \eps \ell\right)\,,
  \end{equation*}
  with
$\|Dv-\Id\|_{L^\infty}\le c\alpha$  and with
 the boundary conditions $v(x,y)=(x,y)$ for $y\in\{y_0,y_0+h\}$ or $x=x_0$,
$v(x_0+\ell,y)=(x_0+\ell+\alpha\vartheta_{h}(y-y_0)),y)$.
The constant $c$ is universal.
\end{lemma}
\begin{proof}
We use the domain subdivision from Lemma \ref{eqconstronerecttheo2bl} with
\begin{equation*}
  v_1(x,y)=
 \begin{cases}
   x+\alpha y & \text{on $\omega_B'$}, \\
 x+\frac{1}{4} \alpha h\gamma\left(\frac{x}{\ell}\right) & \text{on
   $\omega_M'$}, \\ 
 x+\alpha(\frac{h}{2}-y) & \text{on $\omega_A$},\\
 x-\frac{1}{4}\alpha h\gamma\left(\frac{x}{\ell}\right) & \text{on
   $\omega_M''$}, \\ 
 x+\alpha (y-h) & \text{on $\omega_B''$}, 
 \end{cases}
\end{equation*}
and estimate the various energy terms as in Lemma \ref{eqconstronerecttheo1}.
\end{proof}
We remark that in both 
 Lemma \ref{eqconstronerecttheo1bl} and  Lemma \ref{eqconstronerecttheo1}  a
 simpler choice of $\gamma$ (for example, $\gamma(x)=x$) would have been
 possible, since no derivatives of $\gamma$ enter the definition of
 $v$. 
\begin{lemma}\label{lemmaUBK1}
  Under the assumptions of Theorem \ref{Theorem:K1}, there are two functions $u,v\in\calM$
  such that
\begin{equation*}
  E_1^\eps[u,\Omega]\le c \min\Bigl\{\alpha^{4/3} \eps^{2/3} H L^{1/3}+\alpha\eps
  L, \alpha^2 LH\Bigr\}
\end{equation*}
and
\begin{equation*}
  E_1^\eps[v,\Omega]\le c \min\Bigl\{\alpha^4 LH +  \alpha^{4/3} \eps^{2/3} L
  H^{1/3}+\alpha\eps H,
\alpha^2 L H\Bigr\}\,.
\end{equation*}
The constant $c$ is universal.
\end{lemma}
\begin{proof}
The construction of $u$  is identical to the one discussed in Lemma \ref{lemmaUBK2},
with Lemma \ref{eqconstronerecttheo1bl} instead of Lemma
\ref{eqconstronerecttheo2bl}
and  Lemma \ref{eqconstronerecttheo1} instead of Lemma
\ref{eqconstronerecttheo2}, only the choices of $\theta$ and $N$ differ.
In particular, one  obtains
\begin{equation*}
 E_1^\eps[u,\Omega'] \leq c\sum_{i=0}^{\tau}\left(
  \frac{\alpha^2H^3}{2^{2i}N^2\theta^{i}L(1-\theta)} +
  \alpha\varepsilon2^{i}N\theta^{i}(1-\theta)L \right)\,.
\end{equation*}
In this case, the series converges for all $\theta\in (1/4,1/2)$.
We set,  arbitrarily, $\theta=1/3$ and obtain
\begin{equation*}
 E_1^\eps[u,\Omega']\leq c \left(\frac{\alpha^2H^3}{N^2L} + \alpha\varepsilon NL\right) . 
\end{equation*}
We choose  $N=\lceil \alpha^{1/3}H/(\varepsilon^{1/3}L^{2/3})+4H/L\rceil$ and conclude
\[
 E_1^{\varepsilon}[u,\Omega] \leq c \alpha^{4/3}\varepsilon^{2/3}HL^{1/3}+c\alpha\eps (L+H)\,.
\]
The treatment of the bound $\alpha^2LH$ and the elimination of the term
$\alpha \eps H$ are analogous to the other case.

We now turn to $v$. In the case $K_1$ there are two rank-one connections, and if $H$ is much
larger than $L$, it is convenient to use a rotated pattern with long thin stripes along $e_2$. 
This construction can be obtained by rotating the one just
derived. We set
\[
Z=\mii{0}{1}{1}{0}
\]
and identify it with a linear map from $\R^2$ to $\R^2$. 
Let $u$ be the function constructed above, using $\tilde
H=L$, $\tilde L=H$. Then $u$ is the identical map on the boundary of
$\tilde\Omega=(0,H)\times (0,L)$. We set
\begin{equation*}
  v(x,y)=Z u (Z(x,y))\,.
\end{equation*}
Then $v$ obeys the required boundary conditions on $\partial\Omega$ and 
$|D^2v|(\Omega)=|D^2u|(\tilde\Omega)$.  In order to estimate the strain
energy, we observe that $Dv = Z Du (Z(x,y))Z$. We assert that for all $F\in
\R^{2\times 2}$ one has
\begin{equation}\label{eqdistzfza}
  \dist(ZFZ,K)\le \dist(F,K)+\alpha^2\,.
\end{equation}
This  then implies
\begin{equation*}
  E_1^\eps[v,\Omega]\le 2 E_1^\eps[u,\tilde\Omega] + 2\alpha^4 LH\,,
\end{equation*}
as required.

It remains to prove (\ref{eqdistzfza}). 
We first show that there is a rotation $R_a\in \SO(2)$
such that $|R_aZA_1Z-A_1|\le\alpha^2$, and correspondingly for $B_1$.
We set
\begin{equation*}
  R_a=\frac{1}{\sqrt{1+\alpha^2}}\mii{1}{-\alpha}{\alpha}{1}\in\SO(2)
\end{equation*}
and compute
\begin{equation*}
  R_aZA_1Z=\mii{\sqrt{1+\alpha^2}}{-\alpha/\sqrt{1+\alpha^2}}{0}{1/\sqrt{1+\alpha^2}}\,,
\end{equation*}
which gives $|R_aZA_1Z-A_1|\le\alpha^2$ for all $\alpha\in (-1,1)$.
Let now $Q\in \SO(2)$, $F\in \R^{2\times 2}$ be such that
$\dist(F,K)=|F-QA_1|$ (the case with $B_1$ is the same). 
We compute
\begin{alignat*}1
\dist(ZFZ,K)&\le |ZFZ-ZQA_1Z| + \dist(ZQA_1Z, K)\\
&= |F-QA_1| + \dist (ZA_1Z,K)
\le \dist(F,K)+\alpha^2\,,
\end{alignat*}
where we used  that $Z\in \OO(2)\setminus \SO(2)$ implies
$ZQZ\in \SO(2)$, and obtain (\ref{eqdistzfza}). 
\end{proof}

\section{Lower bounds}
\label{seclowerbound}
\subsection{General Strategy}
As in many singularly perturbed nonconvex problems, the 
 lower bound arises by the interaction of the following effects: the second-gradient
 term limits the oscillations of the gradient in the interior of the domain, 
 and therefore brings the deformation $u$ away from the interpolation of the
 boundary data. The strain term then makes a deviation of the deformation
 from the interpolation of the boundary data expensive. The main point in the proof is
 to make these effects  quantitative; in the present case, the main difficulty
 resides in the gradient term, which does not control individual derivatives
 but only the deviation of the  deformation gradient from the set of
 matrices $K_j$. 

 In order to make the strategy quantitative, it is convenient to localize at
 the appropriate length scale. We shall denote by $\lambda$ a length scale
 chosen at the end of the proof. Qualitatively, it can be understood as a
 typical length scale of the microstructure in the interior of the domain. 
 Since energy can concentrate, we shall choose (in Lemma \ref{Lemma:Epart}\ref{Lemma:Epart1})
 a square of side length $\lambda$ as well as a horizontal and a vertical stripe of width
 $\lambda$ on  which the energy is no higher than
on average. Then we shall show  (Lemma \ref{Lemma:Poincare}\ref{Lemma:Poincare1})
 that on this square any low-energy deformation is approximately affine, with
 gradient in $K_j$. In order to relate this to the  boundary data, exploiting
 the elastic energy on the stripes, we need to eliminate rotations from the
 picture. This is done in Lemma
 \ref{lemmaaverage0}, which shows that if a vector has length approximately
 unity, and one component averages (because of the boundary data) to 1, then
 the vector is approximately constant. We then separate the two different
 settings we consider. In the case
$K_1$, addressed in Section  \ref{seccaseK1}, the key remaining difficulty is
the treatment of the two possible orientations for the microstructure, which
requires a separation of two cases. In the case $K_2$, addressed in Section
\ref{seccaseK2}, the key difficulty is in the treatment of the
rotational invariance of the elastic energy and the degeneracy of the rank-one connection. As
we know from the upper bound, see the discussion in (\ref{eqlinearize}), the two
components can interact to substantially reduce the energy. 
Therefore,
an optimal lower
bound will require a careful treatment of this interaction. This
will be done using a suitable test function to integrate twice by parts, see
Lemma \ref{lemmaLBK2} below. In this section, we use $x=(x_1,x_2)$ for an
element of $\R^2$.
 
\begin{lemma}[Localization]\label{Lemma:Epart}\label{Lemma:Poincare}
Assume $H,L>0$, $j\in\{1,2\}$, $u\in \calM$, $\alpha\in (0,1)$, $\eps>0$.
For any $\lambda\in (0,\min\{L,H\}]$, there are stripes
$S=(0,L)\times (s,s+\lambda )\subset\Omega$, $S' = (s',s'+\lambda)\times
(0,H)\subset\Omega$ such that, writing for brevity
$E_j=E_j^\eps[u,\Omega]$, the following holds:
\begin{enumerate}
\item\label{Lemma:Epart1} 
\begin{alignat}1
\label{epart1}
E_j^{\varepsilon}[u,S]&\leq c\frac{\lambda}{H}  E_j, \\
\label{epart2}
E_j^{\varepsilon}[u,S']&\leq c \frac{\lambda}{L}  E_j, \\
\label{epart3}
E_j^{\varepsilon}[u,S\cap S']&\leq c  \frac{\lambda^2}{LH}  E_j,\\
\label{epart4}
\|\partial_1 u_2\|_{L^1(S\cap S')} &\leq c\frac{\lambda^2}{LH}  
\|\partial_1 u_2\|_{L^1(\Omega)}\,.
\end{alignat}
\item\label{Lemma:Poincare1}
There are  $ F\in
K_j$ and $a\in\R^2$ such that
\begin{equation}\label{Poincare}
\int_{S\cap S'}|Du-F|\,dx \leq
c\frac{\lambda^3}{\varepsilon LH } E_j
+ c \frac{\lambda^2}{L^{1/2}H^{1/2}} E_j^{1/2}
\end{equation}
and
\begin{equation}\label{Poincare2}
\int_{S\cap S'}|u(x)-Fx-a|\,dx \leq
c\frac{\lambda^4}{\varepsilon LH } E_j
+ c \frac{\lambda^3}{L^{1/2}H^{1/2}} E_j^{1/2}. 
\end{equation}
\end{enumerate}
The constant $c$ is universal.
\end{lemma}
\begin{figure}[t]
\centering
\includegraphics[width=5cm]{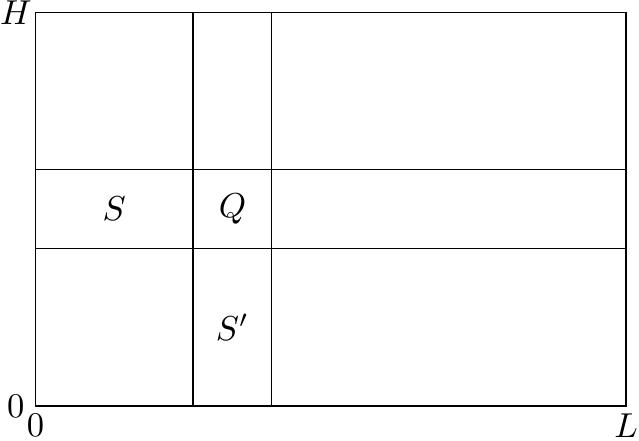}
\caption{Geometry in Lemma \ref{Lemma:Epart}.}\label{Bild:Epart}
\end{figure}

Before starting the proof, we recall the Poincar\'e inequality for $BV$ functions.   We shall use it only on squares, but exploit the explicit dependence of the constant on the size of the square. Precisely, there is a constant $c_P>0$ such that for any square $q_r=x_*+(0,r)^2$ and any function $v\in BV(q_r;\R^m)$ 
one has
\begin{equation}\label{eqpoincare}
  \int_{q_r} |v-\overline v| dx \le c_P \,r\, |Dv|(q_r)\,,
\end{equation}
where $\overline v\in \R^m$ is the average of $v$ over $q_r$. For $r=1$, this assertion is a special case, for example, of \cite[Th. 3.44]{AFP}; 
the same estimate for generic $r$ follows by applying the $r=1$ bound to  $v_r(x)=v(rx)$.
\begin{proof}
\ref{Lemma:Epart1}: 
We  subdivide $\Omega$ into  $M=\lfloor H/\lambda\rfloor\ge H/(2\lambda) $
disjoint horizontal stripes  
$S_k=(0,L)\times(k\lambda, (k+1)\lambda)$, where $\lfloor t\rfloor =
\max\{z\in\Z: z\le t\}$. Since
\begin{equation*}
  \sum_{k=0}^{M-1} E_j^\varepsilon[u, S_k]\le E_j\,,
\end{equation*}
we have
\begin{equation*}
  \# \Bigl\{ k\in [0,M)\cap \N: E_j^\varepsilon[u, S_k]\ge \frac5M E_j \Bigr\}\le \frac{M}{5}
\end{equation*}
(in particular, for $M<5$, the set is empty). 
Analogously, we have
$N=\lfloor L/\lambda\rfloor\ge L/(2\lambda) $
disjoint vertical stripes  
$S_i'=(i\lambda, (i+1)\lambda)\times(0,H)$, and at most  $N/5$ of them do not
satisfy the corresponding bound. 
Since for any of the ``bad'' choices of $k$ all choices of $h$ need to be eliminated, the bound  (\ref{epart1}) is satisfied by all choices of the pair $(k,i)$ except for $NM/5$. Analogously,  (\ref{epart2}) is violated 
for at most $N/5$ choices of $i$, which corresponds to at most $NM/5$ choices of the pair $(k,i)$.
Therefore, the total number of  choices
of $(k,i)$ which do not satisfy (\ref{epart1}) and (\ref{epart2}), with
$c=10$, is  no larger than $2NM/5$. 

A similar estimate leads to
\begin{equation*}
  \# \Bigl\{ (k,i)\in [0,M)\times[0,N)\cap \N^2: E_j^\varepsilon[u, S_k\cap S_i']\ge \frac5{MN} E_j \Bigr\}\le \frac{MN}{5}\,.
\end{equation*}
Hence, for all but 
$MN/5$ of the possible choices (\ref{epart3}) holds, with
$c=20$, and the same for $\|\partial_1 u_2\|_{L^1(S_k\cap S_i')}$. 
Since we have four
conditions, and each of them excludes at most $\lfloor MN/5\rfloor $ choices, there are
necessarily some left, for which all four conditions are valid.
The resulting geometry is illustrated in Figure \ref{Bild:Epart}.

\ref{Lemma:Poincare1}:
Let $F'\in\R^{2\times2}$ be the average of $Du$ over  $Q=S\cap S'$.
Writing $F'=F'-Du+Du$, we obtain  
\begin{alignat*}1
  \lambda^2 \dist(F',K)&\le \|Du-F'\|_{L^1(Q)}+ \|\dist(Du,K)\|_{L^1(Q)}\,.
\end{alignat*}
We define $F$ as the matrix in $K$ closest to $F'$ and 
write $|Du-F|\le |Du-F'|+\dist(F',K)$, so that
\begin{alignat*}1
\|Du-F\|_{L^1(Q)}&\le   \|Du-F'\|_{L^1(Q)}+\lambda^2 \dist(F',K)\\
&\le 2\|Du-F'\|_{L^1(Q)}+ \|\dist(Du,K)\|_{L^1(Q)}\,.
\end{alignat*}
Using Poincar\'e's inequality (\ref{eqpoincare}) on the first term and Hölder's inequality on the
second,  and then recalling 
 (\ref{epart3}), we obtain
\begin{alignat*}1
\|Du-F\|_{L^1(Q)}
&  \le   c\lambda |D^2u|(Q)
+\lambda\|\dist(Du,K)\|_{L^2(Q)}\\
& \le c \frac{\lambda^3}{LH\varepsilon}  E_j +
c  \lambda^2 \frac{E_j^{1/2}}{L^{1/2}H^{1/2}}\,,
\end{alignat*}
which concludes the proof of (\ref{Poincare}). 
To prove (\ref{Poincare2}), it suffices to define $a\in\R^2$ as the average of
$u(x)-Fx$ over $Q$ and apply Poincar\'e's inequality once more.
\end{proof}

The above estimates immediately give the lower bound in the case of very thin
domains. 
\begin{lemma}[Thin domains]\label{lemmathin}
Let $L,H>0$, $\alpha\in (0,1)$, $\eps>0$. 
Then for every $u\in \calM$ we have
$E_j\ge c \min\{\alpha \eps(L+H), \alpha^2 LH\}$.  
The constant $c$ is universal.
\end{lemma}
Before starting, we recall that by the trace theorem in $W^{1,1}$, there is a universal constant $c_T>0$ such that for any
square $q_r=(0,r)^2$,
\begin{equation}\label{eqtrace}
  \int_{\partial q_r} |w| d\calH^1 \le c_T \int_{q_r}\left( \frac1r |w|+|Dw|\right) dx\,.
\end{equation}
The case $r=1$ is given, for example, in 
\cite[Th. 1, Sect. 4.3]{EvansGariepy}, the same estimate for generic $r$ follows by applying the $r=1$ bound to  $w_r(x)=w(rx)$.
\begin{proof}
  We use Lemma \ref{Lemma:Epart} with $\lambda=\min\{L,H\}$ and observe that
  $Q=S\cap S'$ is a square with at least 
  two sides on $\partial\Omega$. By the trace estimate (\ref{eqtrace}) applied
to $w(x)=u(x)-Fx-a$, (\ref{Poincare})--(\ref{Poincare2}) imply
  \begin{equation*}
    \int_{\partial Q} |u(x)-Fx-a|\, d\calH^1(x)\le 
c\frac{\lambda^3}{\varepsilon LH } E_j
+ c \frac{\lambda^2}{L^{1/2}H^{1/2}} E_j^{1/2}
\,.
  \end{equation*}
  On the other hand, we show below that there is a universal constant $c_*>0$
  such that 
  \begin{equation}\label{eqFIdrand}
    c_*\alpha\lambda^2 \le         \int_{\partial Q\cap\partial\Omega}
    |x-Fx-a|\, d\calH^1(x) \text{ for all $a\in \R^2$, $F\in K_j$}\,.
  \end{equation}
Since $u(x)=x$ on $\partial\Omega$, we obtain 
  \begin{equation*}
    c_*\alpha\lambda^2 \le c \frac{\lambda^3}{\varepsilon LH } E_j
+ c \frac{\lambda^2}{L^{1/2}H^{1/2}} E_j^{1/2}\,.
  \end{equation*}
At least one of the two addends must be larger than 
$    c_*\alpha\lambda^2/2$, therefore,
\begin{equation*}
  E_j\ge c \min\Bigl\{\frac{\alpha\varepsilon LH}{\lambda}, \alpha^2 LH\Bigr\}\,,
\end{equation*}
inserting $\lambda$ this  implies  the assertion.

It remains to prove (\ref{eqFIdrand}). 
Let one of the sides of $Q$ 
 contained in $\partial\Omega$ be of the form
 $q+(0,\lambda)e_i$, $i=1$ or $2$.
Since for any $v\in \R^2$, we have
\begin{equation*}
  \min_{a'\in \R^2} \int_0^\lambda |tv+a'| dt =\frac14  \lambda^2|v|\,,
\end{equation*}
we obtain (with $a'=q-Fq-a$, $v=e_i-Fe_i$)
\begin{equation*}
  \frac14\lambda^2 |Fe_i-e_i|\le\int_0^\lambda | (\Id-F)(q+te_i)-a| dt\,.
\end{equation*}
If there are two orthogonal sides of $\partial Q$ 
contained in $\partial\Omega$,
this gives control of the norm of $F-\Id$. If instead only two 
parallel sides of $\partial Q$  are contained in $\partial\Omega$, 
we need one more step. We assume for notational simplicity that they
are both horizontal, so that the previous equation gives control of $|Fe_1-e_1|$. 
We write
\begin{alignat*}1
 \lambda^2 |Fe_2-e_2|&=
  \int_0^\lambda |(\Id-F)\lambda e_2|dt\\
&\le
  \int_0^\lambda |(\Id-F)te_1-a|dt+\int_0^\lambda|(\Id-F)(te_1+\lambda e_2)-a| dt \,.
\end{alignat*}
The first integral corresponds to the bottom side, the last one to the top side.
In both cases, we conclude
\begin{equation*}
  \frac14\lambda^2 (|Fe_1-e_1|+|Fe_2-e_2|)   \le      2 
  \int_{\partial Q\cap\partial\Omega}     |x-Fx-a|\, d\calH^1(x)\,.
\end{equation*}
Finally, if $F\in K_2$, we have $|Fe_2|=1\pm \alpha$, which implies
$|Fe_2-e_2|\ge \alpha$ and concludes the proof of (\ref{eqFIdrand}) with $c_*=1/8$.
If $F\in K_1$, then
\begin{alignat*}1
  \alpha&=|Fe_1\cdot Fe_2| = | (Fe_1-e_1)\cdot Fe_2 + e_1\cdot (Fe_2-e_2)|
\\ &  \le 2 |Fe_1-e_1|+|Fe_2-e_2|\,,
\end{alignat*}
and the proof of (\ref{eqFIdrand})  is concluded with $c_*=1/16$.
\end{proof}

Before closing this general part, we present a lemma which permits to deal with
the boundary data.  
\begin{lemma}\label{lemmaaverage0}
  Let $\omega\subset\R^n$ be bounded and measurable, $v\in L^1(\omega;\R^2)$,
  $d\in L^2(\omega;[0,\infty))$ and $e\in 
  S^1$ be such  that $v\cdot e-1$ has average 0 and $|v|\le 1+d$ almost
  everywhere. Then 
  \begin{equation*}
    \|v\cdot e-1\|_{L^1}\le 2 |\omega|^{1/2}\|d\|_{L^2}
  \end{equation*}
  and
  \begin{equation*}
    \|v\cdot e^\perp\|_{L^1}\le
     3 |\omega|^{3/4}\|d\|_{L^2}^{1/2}+|\omega|^{1/2}\|d\|_{L^2}\,.
  \end{equation*}
All norms are taken over $\omega$.
\end{lemma}
\begin{proof}
  We assume $e=e_1$ without loss of generality. Since $v_1-1$ has average 0,
  we have (writing $f_\pm=\max\{\pm f,0\}$)
  \begin{alignat}1\label{eqv1m1}
    \|v_1-1\|_{L^1}&=\int_\omega \left[(v_1-1)_++(v_1-1)_-\right] dx\\
& = 2\int_\omega(v_1-1)_+\, dx\le 2 \int_\omega (|v|-1)_+ dx
    \le 2 \|d\|_{L^1}\,.  \nonumber
  \end{alignat}
Using Hölder's inequality, this leads directly to the first 
bound given in the statement. To address the second one, we write
$v_2^2=|v|^2-v_1^2$ 
which implies, recalling
 $|v|\le 1+d$,
\begin{equation}\label{eqv2dv}
  \|v_2\|_{L^2}^2 = \|v\|_{L^2}^2-\|v_1\|_{L^2}^2
\le \|(1+d)^2\|_{L^1} -\|v_1\|_{L^2}^2\,.
\end{equation}
In order to estimate the last term, we write $1=v_1+(1-v_1)$ and use
 a triangular inequality and (\ref{eqv1m1}),
\begin{equation*}
  |\omega|\le \|v_1\|_{L^1}+\|1-v_1\|_{L^1} \le \|v_1\|_{L^1} + 2 \|d\|_{L^1}\,.
\end{equation*}
Hölder's inequality then gives
\begin{equation*}
  |\omega|-2\|d\|_{L^1} \le \|v_1\|_{L^1}\le |\omega|^{1/2}\|v_1\|_{L^2}\,.
\end{equation*}
If $2\|d\|_{L^1}\le |\omega|$, squaring and inserting in (\ref{eqv2dv}), we obtain
\begin{alignat*}1
  \|v_2\|_{L^2}^2 \le& \|(1+d)^2\|_{L^1} -(|\omega|^{1/2}-2
  |\omega|^{-1/2}\|d\|_{L^1})^2 \\
\le& |\omega|+2\|d\|_{L^1}+\|d\|_{L^2}^2 -|\omega|+4\|d\|_{L^1}-
4|\omega|^{-1}\|d\|_{L^1}^2 \\
\le & 6\|d\|_{L^1}+\|d\|_{L^2}^2 \,.
\end{alignat*}
In the case  $2\|d\|_{L^1}\ge |\omega|$, we instead write, again from  (\ref{eqv2dv}),
\begin{equation*}
    \|v_2\|_{L^2}^2 \le \|(1+d)^2\|_{L^1} \le
    |\omega|+2\|d\|_{L^1}+\|d\|_{L^2}^2
    \le 4\|d\|_{L^1}+\|d\|_{L^2}^2\,.
\end{equation*}
In both cases, Hölder's inequality gives
\begin{equation*}
    \|v_2\|_{L^2}^2\le  6\|d\|_{L^1}+\|d\|_{L^2}^2 
    \le 6|\omega|^{1/2}\|d\|_{L^2}+\|d\|_{L^2}^2    \,.  
\end{equation*}
Another application of Hölder's inequality and of the estimate
$(x+y)^{1/2}\le x^{1/2}+y^{1/2}$ gives, rounding for brevity $\sqrt 6\le 3$, 
\begin{alignat*}1
    \|v_2\|_{L^1}&\le |\omega|^{1/2} \|v_2\|_{L^2}\\
&\le
    |\omega|^{1/2}( 6|\omega|^{1/2}\|d\|_{L^2}+\|d\|_{L^2}^2 )^{1/2}
    \le 3 |\omega|^{3/4}\|d\|_{L^2}^{1/2}+|\omega|^{1/2}\|d\|_{L^2}\,.
\end{alignat*}
\end{proof}

\subsection{Proof of the lower bound in Theorem \ref{Theorem:K1}}
\label{seccaseK1}
At this point, we specialize to the first case.
We write for brevity $E=E_1=E_1^\eps[u,\Omega]$ and $K=K_1$.

\begin{lemma}\label{Lemma:dux}
Assume $H,L>0$, $\alpha\in (0,1)$, $\eps>0$, $u\in \calM$. 
Let $S,S'$ be as in Lemma \ref{Lemma:Poincare}. Then
\begin{equation}\label{u1x1bA}
\int_{S\cap S'}|u_1(x)-x_1|\,dx \leq c 
\lambda^2 \frac{L^{1/2}}{H^{1/2}} E^{1/2}
\end{equation}
and
\begin{equation}\label{u2x2bA}
\int_{S\cap S'}|u_2(x)-x_2|\,dx \leq c 
\lambda^2  \frac{H^{1/2}}{L^{1/2}} (E+\alpha^4LH)^{1/2}\,.
\end{equation}
Further,
\begin{equation}\label{eqpartial1u2}
  \|\partial_1 u_2\|_{L^1(\Omega)} \le
  c (HL)^{3/4} E^{1/4}+ c
(HL)^{1/2} E^{1/2}\,.
\end{equation}
The constant $c$ is universal.
\end{lemma}
\begin{proof}
This follows from repeated application of Lemma \ref{lemmaaverage0}. We 
first use $v=\partial_1 u$,  $e=e_1$, $d=\dist(Du,K)$, $\omega=S$. By the boundary
conditions, $\int_S (\partial_1u_1-1)dx=0$, and $|v|\le 1+d$ because $|Fe_1|=1$ for
all $F\in K$ and $|(Du)e_1|\le \min_{F\in K}|Du-F|+|Fe_1|= 1+\dist(Du,K)$. We
obtain, using (\ref{epart1}),
\begin{equation*}
  \|\partial_1 u_1-1\|_{L^1(S)} \le
 c (\lambda L)^{1/2} 
\left(\frac{\lambda E}{H}\right)^{1/2}
  = c \lambda \frac{L^{1/2}E^{1/2}}{H^{1/2}}\,.
\end{equation*}
Integrating, 
\begin{equation*}
  \|u_1-x_1\|_{L^1(S\cap S')}\le \lambda \|\partial_1u_1-1\|_{L^1(S)}
\end{equation*}
proves (\ref{u1x1bA}).
Using $\omega=\Omega$ with the same $v$, $e$, and $d$, we get
\begin{equation*}
  \|\partial_1 u_2\|_{L^1(\Omega)} \le
  3 (H L)^{3/4}  E^{1/4}+ (H
  L)^{1/2} 
E^{1/2}\,,
\end{equation*}
which  proves (\ref{eqpartial1u2}).
We now turn to the other component, and set $v=\partial_2 u$, $e=e_2$,
$d=\dist(Du,K)+\alpha^2$, $\omega=S'$. 
Since
 for all $F\in K$, we have $|Fe_2|=\sqrt{1+\alpha^2}\leq
 1+\alpha^2/2$, we obtain 
$|v|=|(Du)e_2|\leq 1+ \frac12\alpha^2+ \dist(Du,K)$. 
In this case, recalling (\ref{epart2}),
\begin{equation*}
  \|d\|_{L^2(S')}^2 \le c \frac{\lambda E}L +2\lambda H\alpha^4
  \le c \frac{\lambda (E+\alpha^4LH)}{L}\,.
\end{equation*}
Proceeding as above,  Lemma \ref{lemmaaverage0} gives
\begin{equation*}
  \|\partial_2 u_2-1\|_{L^1(S')} \le
 c (\lambda H)^{1/2} 
\left(\frac{\lambda (E+\alpha^4LH)}{L}\right)^{1/2}\,,
\end{equation*}
 and therefore  (\ref{u2x2bA}).
\end{proof}
\begin{lemma}\label{lemmaLBK1}
Under the assumptions of Theorem \ref{Theorem:K1} one
has,  for any $u\in \calM$,
\begin{alignat*}1
E_1^\eps[u,\Omega]\ge c \min\Bigl\{ & \alpha^{4/3}\varepsilon^{2/3} L^{1/3}H +\alpha \eps L, \\
&\alpha^{4/3}\varepsilon^{2/3} LH^{1/3} +\alpha \eps H 
+ \alpha^4LH ,\alpha^2 LH\Bigr\}\,.
\end{alignat*}
The positive constant $c$ is  universal.
\end{lemma}
\begin{proof}
We fix $\lambda\in(0,\min\{L,H\}]$, the precise value will be chosen below,
and  choose the stripes $S$ and $S'$ as in  Lemma \ref{Lemma:Epart}.

We first show that there is $c_*>0$ such that for every $F\in K$ and 
$a\in \R^2$, one has 
\begin{equation}\label{K1cases}
\begin{split}
&\text{if $|F_{21}|<\alpha/2$ (case 1):} \quad c_*\alpha\lambda^3 \leq \norm{(x-Fx-a)\cdot e_1}_{L^1(S\cap S')}, \\
&\text{if $|F_{21}|\geq\alpha/2$ (case 2):} \quad c_*\alpha\lambda^3  \leq \norm{(x-Fx-a)\cdot e_2}_{L^1(S\cap S')}.
\end{split}
\end{equation}
Both inequalities follow from the fact that for all $\xi\in\R$ one has
(as in the proof of (\ref{eqFIdrand}) above)
\begin{equation*}
  \min_{\eta\in\R} \int_0^\lambda |\xi t + \eta| dt
  = \int_0^\lambda |\xi t - \frac12 \xi \lambda| dt = \frac14 \lambda^2 |\xi|\,.
\end{equation*}
To prove the first inequality in (\ref{K1cases}), we  estimate
\begin{alignat*}1
 \norm{(x-Fx-a)\cdot e_1}_{L^1(S\cap S')}
& = \int_s^{s+\lambda} \int_{s'}^{s'+\lambda} |(x_1-F_{11}x_1-a_1)-F_{12}x_2| dx_2 dx_1\\
&\ge \int_s^{s+\lambda} \min_{\eta\in\R} \int_{s'}^{s'+\lambda} |\eta-F_{12}x_2| dx_2 dx_1
\ge \frac14 \lambda^3 |F_{12}|\,.
\end{alignat*}
Since $F\in K$ implies
$F_{21}=\sin\varphi$, $F_{12}=\pm\alpha\cos\varphi-\sin\varphi$ for some
$\varphi\in\R$,  in the first case, we have $|F_{12}|\ge \alpha
(\sqrt{3/4}-1/2)\ge c\alpha$, and the proof is concluded.
The second case is simpler,
\begin{alignat*}1
 \norm{(x-Fx-a)\cdot e_2}_{L^1(S\cap S')}
& =  \int_{s'}^{s'+\lambda} \int_s^{s+\lambda} |(x_2-F_{22}x_2-a_2)-F_{21}x_1| dx_1 dx_2\\
&\ge \int_{s'}^{s'+\lambda}  \min_{\eta\in\R} \int_{s}^{s+\lambda} |\eta-F_{21}x_1| dx_1 dx_2
\ge \frac14 \lambda^3 |F_{21}|\,.
\end{alignat*}
This concludes the proof of (\ref{K1cases}).

For $i\in\{1,2\}$, we write
\begin{equation*}
|(x-Fx-a)\cdot e_i| \leq |(u(x)-x)\cdot e_i| + |u(x)-Fx-a|\,.
\end{equation*}
In case $i$, we obtain 
\begin{equation}\label{K1cases2}
c_*\alpha \lambda^3  \leq \|u_i(x)-x_i\|_{L^1(S\cap S')} + \|u(x)-Fx-a\|_{L^1(S\cap S')}\,.
\end{equation}
At this point, we  distinguish the two cases.

{\bf Case 1 ($|F_{21}|<\alpha/2$).} 
We estimate the first term in (\ref{K1cases2}) by 
(\ref{u1x1bA}) and the second by (\ref{Poincare2}),
\begin{equation}\label{eqcase1ineq12}
c_*\alpha\lambda^3\leq 
c\lambda^2\frac{L^{1/2}}{H^{1/2}} E^{1/2}
+c\frac{\lambda^4}{\varepsilon LH } E
+ c \frac{\lambda^3}{L^{1/2}H^{1/2}} E^{1/2}
 \,.
\end{equation}
Since $\lambda\le L$, the last term can be absorbed into the first
one, and 
\begin{alignat}1\label{eqcase1ineq1}
c_*\alpha\lambda^3\leq &
c\lambda^2 \frac{L^{1/2}}{H^{1/2}} E^{1/2}
+c\frac{\lambda^4}{\varepsilon LH } E\,,
\end{alignat}
which can be rewritten as
 \begin{equation*}
E\ge c \min\Bigl\{ \frac{\alpha^2 H\lambda^2} {L} , \frac{\alpha\varepsilon LH}{\lambda}   \Bigr\}\,.
\end{equation*}
The right-hand side is maximized by choosing
 $\lambda=(\varepsilon/\alpha)^{1/3}L^{2/3}$.  Since we can 
only choose  $\lambda\in (0,\min\{L,H\}]$, we set
\begin{equation*}
  \lambda=\min\left\{(\varepsilon/\alpha)^{1/3}L^{2/3}, L, H\right\} 
\end{equation*}
and obtain
 \begin{equation*}
E\ge c \min\Bigl\{ \alpha^{4/3}\varepsilon^{2/3} L^{1/3}H   ,
\alpha^2 LH,
\frac{\alpha^2 H^3}{L}\Bigr\}\,.
\end{equation*}
We recall that by Lemma \ref{lemmathin}, we have
$E\ge c \min\{\alpha \eps(L+H), \alpha^2 LH\}$.  
Either $E\ge\alpha^2 LH$ or $E$ is larger than the first expression. Therefore,
we can combine the two estimates to produce
 \begin{equation*}
E\ge c \min\Bigl\{ \alpha^{4/3}\varepsilon^{2/3} L^{1/3}H +\alpha \eps(L+H)  ,
\frac{\alpha^2 H^3}{L}+\alpha \eps(L+H), \alpha^2 LH\Bigr\}\,.
\end{equation*}
We show that the second entry in the minimum can be dropped.  Indeed,
by the arithmetic-geometric mean inequality,
\begin{equation*}
  \alpha^{4/3}\varepsilon^{2/3} L^{1/3}H =
\left(\frac{\alpha^2 H^3}{L}\right)^{1/3}\left(\alpha \eps L\right)^{2/3}
\le \frac13 \, \frac{\alpha^2 H^3}{L} + \frac23\, \alpha \eps L\,.
\end{equation*}
Therefore, the second term  is larger than the first one and is never the minimum.
Analogously, if $\alpha\eps H\ge \alpha^{4/3}\varepsilon^{2/3} L^{1/3}H$, then
$\eps\ge \alpha L$, but in this case, the inimum equals $\alpha^2
LH$. Therefore, we can drop the addend $\alpha\eps H$ in the first term.
We conclude that in  case 1,
 \begin{equation}\label{eqlastboundcase1}
E\ge c \min\Bigl\{ \alpha^{4/3}\varepsilon^{2/3} L^{1/3}H +\alpha \eps L, \alpha^2 LH\Bigr\}\,.
\end{equation}

{\bf Case 2 ($|F_{21}|\ge\alpha/2$).} If 
$\|\partial_1 u_2\|_{L^1(Q)} \le\lambda^2\alpha/4$, then
  (\ref{Poincare}) gives
  \begin{equation*}
    \frac14\lambda^2\alpha\le 
c\frac{\lambda^3}{\varepsilon LH } E
+ c \frac{\lambda^2}{L^{1/2}H^{1/2}} E^{1/2}\,,
  \end{equation*}
which implies (\ref{eqcase1ineq12}),  hence this case has already been treated.
Otherwise, from (\ref{epart4}) and (\ref{eqpartial1u2}) we immediately obtain
\begin{equation*}
  \frac14\alpha\lambda^2 \le c\frac{\lambda^2}{LH} 
  \|\partial_1 u_2\|_{L^1(\Omega)} \le
  c  \lambda^2   \left(\frac{ E}{LH}\right)^{1/4}+ c
\lambda^2
\left(\frac{E}{LH}\right)^{1/2}\,.
\end{equation*}
If $E\ge LH$, there is nothing to prove, hence we can ignore the second
term. Otherwise,
\begin{equation}\label{eqcase2alpha4}
  E\ge c\alpha^4 LH\,,
\end{equation}
which in particular permits to estimate
\begin{equation}\label{eqcase2alpha4b}
  E+\alpha^4LH \le c E\,.
\end{equation}
Using  (\ref{u2x2bA}) and (\ref{Poincare2}) in  (\ref{K1cases2}) yields
\begin{alignat*}1
c_*\alpha\lambda^3\leq &
c\lambda^2
  \frac{H^{1/2}}{L^{1/2}} (E+\alpha^4 LH)^{1/2}
+c\frac{\lambda^4}{\varepsilon LH } E
+ c \frac{\lambda^3}{L^{1/2}H^{1/2}} E^{1/2}\,.
\end{alignat*}
The last term can be absorbed into the first one,  using
(\ref{eqcase2alpha4b})  we obtain
\begin{alignat*}1
c_*\alpha\lambda^3\leq &
c\lambda^2  \frac{H^{1/2}}{L^{1/2}} E^{1/2}
+c\frac{\lambda^4}{\varepsilon LH } E
\end{alignat*}
for all $\lambda\in(0,\min\{L,H\}]$.
This is the same as (\ref{eqcase1ineq1}) up to swapping $L$ and $H$, and gives
 \begin{equation*}
E\ge c \min\Bigl\{ \alpha^{4/3}\varepsilon^{2/3} H^{1/3}L  ,
\frac{\alpha^2 L^3}{H}, \alpha^2 LH\Bigr\}\,.
\end{equation*}
Recalling Lemma \ref{lemmathin},
 \begin{equation*}
E\ge c \min\Bigl\{ \alpha^{4/3}\varepsilon^{2/3} H^{1/3}L +\alpha \eps(L+H)  ,
\frac{\alpha^2 L^3}{H}+\alpha \eps(L+H), \alpha^2 LH\Bigr\}\,,
\end{equation*}
and we drop the same irrelevant  terms as in the previous case.
Recalling (\ref{eqcase2alpha4}),
 we conclude
 \begin{equation*}
E\ge c \min\Bigl\{ \alpha^{4/3}\varepsilon^{2/3} H^{1/3}L +\alpha \eps H 
+ \alpha^4LH , \alpha^2 LH\Bigr\}\,,
\end{equation*}
which together with  (\ref{eqlastboundcase1}) concludes the proof.
\end{proof}

\subsection{Proof of the lower bound in Theorem \ref{Theorem:K2}}
\label{seccaseK2}
We now come to the proof of the lower bound in Theorem \ref{Theorem:K2}. 
In this case the local structure leads to oscillations in $\partial_2u_2$, and
the branching can only be horizontal. 

There are two  ways of producing a lower bound on the energy. The first one is to estimate
$\partial_1u_2$ using Lemma \ref{lemmaaverage0}, 
as was done in \cite{ChanContiSFBBericht}. This gives the optimal $\eps^{4/5}$
scaling, but not the optimal scaling in the parameter  $\alpha$. Indeed, this
argument is based on a purely nonlinear effect, and would not produce any bound
in the linear setting. 

We show here that the bound from \cite{ChanContiSFBBericht} can be
improved. The key idea is to use {\em two} partial integrations, to pass from
the $\partial_2u_2$ term to the $\partial_1 u_1$ one, and to use
for the leading-order term the estimate for $\partial_1 u_1$ in  Lemma
\ref{lemmaaverage0}, 
which scales as $E$ instead of $E^{1/2}$.
Swapping the indices between
the partial derivative and the component of $u$ is done using the invariance
under rotations. In order to be able to do the repeated partial integration, it
is helpful to test the derivatives with a smooth test function, instead of
simply integrating over a suitable domain, as done in Section \ref{seccaseK1}.

The key result in this section, which proves the lower bound of Theorem
\ref{Theorem:K2},  is the following. 
\begin{lemma}\label{lemmaLBK2}
Under the assumptions of Theorem \ref{Theorem:K2}, one
has
  \begin{equation*}
    \inf \{ E_2^\eps[u,\Omega]:u\in\calM\} \ge c\min\Bigl\{\alpha^{6/5}\varepsilon^{4/5}L^{1/5}H +
    \alpha \eps
    L, \alpha^2 LH\Bigr\}\,.
  \end{equation*}
The positive constant $c$ is universal.
\end{lemma}
\begin{proof}
Since the proof involves a treatment of the value of $Du$ on segments, to avoid a notationally complex treatment of traces we first use density to show that it suffices to treat smooth  functions.
Let $u\in\calM$, $\delta>0$.  We extend $u$ by $u(x)=x$ on $\R^2\setminus\Omega$.  Since $\Omega$ is convex, by scaling we can find $\tilde u\in\calM$ such that 
$u(x)=x$ in a neighbourhood of $\partial\Omega$ and
$E_2^\eps[\tilde u,\Omega]\le 
E_2^\eps[ u,\Omega]+\delta$.
For $\rho>0$, we 
let $u_\rho\in C^\infty(\R^2;\R^2)$ be a mollification of $\tilde u$ on a scale $\rho$. For sufficiently small $\rho$, one has $u_\rho\in\calM$.
At the same time, for $\rho\to0$, one has $|D^2u_\rho|(\Omega)\to |D^2\tilde u|(\Omega)$ and
$Du_\rho\to D\tilde u$ in $L^2(\Omega;\R^{2\times2})$. Therefore,
 there is $\rho_\delta>0$ such that
\begin{equation*}
 E_2^\eps[u_{\rho_\delta},\Omega] \le  E_2^\eps[\tilde u,\Omega]+\delta\le 
E_2^\eps[ u,\Omega]+2\delta\,.
\end{equation*}
We conclude that
\begin{equation*}
    \inf \{ E_2^\eps[u,\Omega]:u\in\calM\} =
    \inf \{ E_2^\eps[u,\Omega]:u\in\calM\cap C^\infty(\overline\Omega;\R^2)\} \,.
  \end{equation*}

Let now $u\in\calM\cap C^\infty(\overline\Omega;\R^2)$, this function 
will be fixed for the rest of the proof.
Let $\lambda\in(0,\min\{L,H\}]$, the precise value will be chosen below. 
We choose the stripes $S$ and $S'$ with Lemma \ref{Lemma:Epart} and set
$Q=S\cap S'$, $I=(s,s+\lambda)$.
Let $F$ be as in Lemma \ref{Lemma:Poincare}\ref{Lemma:Poincare1}.
By  (\ref{Poincare}) there is $x_1^*$ such that the segment 
 $I^*=\{x_1^*\}\times I\subset Q$ has the property
\begin{equation}\label{eqDuFI}
  \int_{I} |Du-F|(x_1^*,t)dt \le \frac{1}{\lambda} \int_Q |Du-F|dx \le
  c\frac{\lambda^2}{\eps LH} E + c\frac{\lambda E^{1/2}}{L^{1/2}H^{1/2}}\,.
\end{equation}

We exploit the boundary conditions, using again Lemma
\ref{lemmaaverage0}.
In particular, we set $v=\partial_1u$, $e=e_1$,
  $d=\dist(Du,K)$, $\omega=S$ and obtain, recalling (\ref{epart1}),
\begin{equation}\label{eqK2d1u1}
  \|\partial_1 u_1-1\|_{L^1(S)} \le
 c \lambda \frac{L^{1/2}E^{1/2}}{H^{1/2}}
\end{equation}
and
  \begin{equation*}
    \|\partial_1u_2\|_{L^1(S)}\le 
  c \lambda \frac{L^{3/4} E^{1/4}}{H^{1/4}}+ c \lambda 
\frac{L^{1/2} E^{1/2}}{H^{1/2}}\,.
  \end{equation*}
Since, in the case $E\ge HL$, there is nothing to prove, we can replace the last
estimate with the simpler one
  \begin{equation}\label{eqK2d1u2}
    \|\partial_1u_2\|_{L^1(S)}\le 
  c \lambda \frac{L^{3/4} E^{1/4}}{H^{1/4}}\,.
  \end{equation}
This will only be used in the estimate of the error term arising from the
linearization of $\SO(2)$, in the second line of (\ref{eqcase1k2a}). The
leading-order terms will be estimated with (\ref{eqK2d1u1}).

At this point, the proof starts to differ significantly from the one for
$K_1$.
We fix one test function $\varphi\in C^\infty_c([0,1];[0,1])$ with $\varphi=
1$ on $(1/4,3/4)$, and define $f\in C^\infty_c(I)$ by
 $f(\lambda t+s)=\varphi(t)$, so that
\begin{equation*}
\norm{f}_{L^1(I)} \ge\frac\lambda2 , \quad \norm{f}_{L^{\infty}}=1, \quad
 \norm{f'}_{L^{\infty}}\le\frac{c}{\lambda},  \quad \norm{f''}_{L^\infty}\le\frac{c}{\lambda^2}\,.
\end{equation*}
We distinguish two cases, depending on the value of $F_{22}$. 

{\bf Case 1: $|F_{22}-1|\geq\alpha/2$}. 
With a triangular inequality and a partial integration we write
\begin{alignat}1 \nonumber
  \frac\alpha4\lambda\le \frac\alpha2\|f\|_{L^1(I)}
&\le \left| \int_I (F_{22}-1) f dx_2\right|\\  \nonumber
&\le \left| \int_I (F_{22}-\partial_2u_2) f dx_2\right|
+\left| \int_I (\partial_2u_2-1) f dx_2\right|\\ \nonumber
&\le  \int_I |Du-F| dx_2 
+\left| \int_I (u_2-x_2) f'(x_2) dx_2\right|\\
&\le  \int_I |Du-F| dx_2 
+\left| \int_{(0,x_1^*)\times I}\partial_1u_2(x) f'(x_2) dx\right|\,,
\label{eqcase1k2a}
\end{alignat}
where we used $\|f\|_{L^\infty}=1$ and in the last step the boundary data
 $u_2(0,x_2)=x_2$. In all integrals over $I$, the functions $u$ and $Du$ are
 evaluated at $(x_1^*,x_2)$.
At this point we trade $\partial_1 u_2$ with $\partial_2 u_1$ in the last term. To do this, we
use the structure of rotations.
First, we choose $G\in L^\infty(S;K)$ with 
$|Du-G|=\dist(Du,K)$ almost everywhere, then choose 
 $\sigma \in L^\infty(S;\{-1,1\})$ such that
$G\in\SO(2)\text{diag}(1,1+\sigma\alpha)$ almost everywhere. 
A direct computation shows that  $G_{12}+(1+\sigma\alpha)G_{21}=0$, and therefore
\begin{equation*}
 (1+\sigma\alpha)\partial_1u_2+\partial_2u_1 =
 (1+\sigma\alpha)(\partial_1u_2-G_{21})+ (\partial_2u_1-G_{12}) \,.
\end{equation*}
We multiply this expression with $f'(x_2)$ and integrate it over $(0,x_1^*)\times
I$ to obtain
\begin{alignat}1 \nonumber
 \left|\int_{(0,x_1^*)\times I}\partial_1u_2(x)f'(x_2)\,dx \right|
\le&  
  \left|\int_{(0,x_1^*)\times I}\partial_2u_1(x)f'(x_2)\,dx \right|\\
&+ \alpha  \nonumber
\|f'\|_{L^\infty}\int_{(0,x_1^*)\times I} |\partial_1u_2(x)| dx \\
&+ 3\int_S |f'|(x_2)\, |Du-G|(x)\, dx\,.
\label{eqk2lonif1}
\end{alignat}
We treat the different terms separately. The term with $\partial_2 u_1$ is
integrated by parts and then estimated using (\ref{eqK2d1u1}),
\begin{alignat*}1
\left|\int_{(0,x_1^*)\times I}\partial_2u_1(x)f'(x_2)\,dx \right|&=
\left|\int_{(0,x_1^*)\times I}(u_1(x)-x_1)f''(x_2)\,dx \right|\\
&\le \|f''\|_{L^\infty} \|u_1-x_1\|_{L^1(S)} \\
&\le L \|f''\|_{L^\infty} \|\partial_1u_1-1\|_{L^1(S)} \\
&\le L \frac{c}{\lambda^2} 
 c \lambda \frac{L^{1/2}E^{1/2}}{H^{1/2}}
=  \frac{c}{\lambda} 
  \frac{L^{3/2}E^{1/2}}{H^{1/2}}\,.
\end{alignat*}
For the second term, we use  $ \|f'\|_{L^\infty}\le c/\lambda$
and (\ref{eqK2d1u2}) to obtain
\begin{alignat*}1
  \alpha \|f'\|_{L^\infty} \int_{(0,x_1^*)\times I} |\partial_1u_2|(s,x_2)\,dx
&\le
c\alpha \frac{L^{3/4} E^{1/4}}{H^{1/4}}\,.
\end{alignat*}
Finally, the last term  is bounded, using
(\ref{epart1}),
\begin{alignat*}1
  \int_S |f'| |Du-G| dx\le \|f'\|_{L^\infty} \int_S \dist(Du,K)dx \le c \frac{L^{1/2} E^{1/2}}{H^{1/2}}\,. 
\end{alignat*}
The estimate (\ref{eqcase1k2a}) becomes,
using (\ref{eqDuFI}), (\ref{eqk2lonif1}) and the three previous estimates,
\begin{equation}\label{eqk2alletermsad}
  \alpha\lambda \le
  c\frac{\lambda^2}{\eps LH} E + c\frac{\lambda E^{1/2}}{L^{1/2}H^{1/2}}
+ \frac{c}{\lambda}    \frac{L^{3/2}E^{1/2}}{H^{1/2}}
+ c\alpha \frac{L^{3/4} E^{1/4}}{H^{1/4}}
+ c \frac{L^{1/2} E^{1/2}}{H^{1/2}}  
\,.
\end{equation}
Since $\lambda\le L$,  the second term and the last one can be absorbed into the
third one, and we have
\begin{equation*}
  \alpha\lambda \le
  c\frac{\lambda^2}{\eps LH} E 
+ \frac{c}{\lambda}    \frac{L^{3/2}E^{1/2}}{H^{1/2}}
+ c\alpha \frac{L^{3/4} E^{1/4}}{H^{1/4}}
\end{equation*}
for all $\lambda\in (0,\min\{L,H\}]$. Therefore,
\begin{equation*}
  E\ge c \min\Bigl\{ \frac{\alpha\eps LH}{\lambda}, \frac{\alpha^2 \lambda^4 H}{L^3}, 
  \frac{\lambda^4 H}{L^3}\Bigr\}\,.
\end{equation*}
Since $\alpha\le 1$, the last term can be ignored. Finally, we choose $\lambda$
as
\begin{equation*}
  \lambda=\min\Bigl\{ \frac{\eps^{1/5} L^{4/5} }{\alpha^{1/5}}, L,H\Bigr\}
\end{equation*}
and obtain
\begin{equation*}
  E\ge c\min\Bigl\{  \alpha^{6/5} \eps^{4/5} L^{1/5} H, \alpha^2 LH, \frac{\alpha^2 H^5}{L^3}\Bigr\}\,.
\end{equation*}
By Lemma \ref{lemmathin}, we  have 
$E\ge c \min\{\alpha \eps(L+H), \alpha^2 LH\}$.  
Therefore,
\begin{equation*}
  E\ge c\min\Bigl\{  \alpha^{6/5}\eps^{4/5} L^{1/5} H+\alpha\eps(L+H), \alpha^2 LH,
  \frac{\alpha^2 H^5}{L^3}+\alpha\eps(L+H)\Bigr\} \,.
\end{equation*}
Since, by the arithmetic-geometric mean inequality,
\begin{equation*}
\eps^{4/5}
\alpha^{6/5} L^{1/5} H=
\left(\frac{\alpha^2 H^5}{L^3}\right)^{1/5}\left(\alpha\eps L\right)^{4/5}
\le\frac15\,\frac{\alpha^2 H^5}{L^3} +\frac45\,\alpha\eps L\,,
\end{equation*}
the last term is always larger than the first one and
can be dropped. Further, if 
$\eps\alpha H\ge\alpha^{6/5} \eps^{4/5} L^{1/5} H$, then 
$\eps\ge \alpha L$, and therefore the relevant energy bound is $\alpha^2 LH$. 
We conclude
\begin{equation*}
  E\ge c\min\Bigl\{\alpha^{6/5} \eps^{4/5}  L^{1/5} H+\alpha\eps L, \alpha^2 LH\Bigr\}\,,
\end{equation*}
and therefore the proof in this case.

{\bf Case 2: $|F_{22}-1|<\alpha/2$}. 
Since $|Fe_2|=1\pm \alpha$, in this case necessarily $|F_{12}|\ge \alpha/2$. 
Therefore, proceeding as above,
\begin{alignat*}1
  \frac14\alpha\lambda\le \frac12\alpha \|f\|_{L^1(I)}
&= \left| \int_I F_{12} f(x_2) dx_2\right|\\
&\le \left| \int_I (F_{12}-\partial_2u_1) f(x_2) dx_2\right|
+\left| \int_I \partial_2u_1 f(x_2) dx_2\right|\\
&\le  \int_I |Du-F| dx_2 
+\left| \int_I (u_1-x_1) f'(x_2) dx_2\right|\,.
\end{alignat*}
In turn, recalling (\ref{eqK2d1u1}),
\begin{equation*}
\left| \int_I (u_1-x_1) f'(x_2) dx_2\right|\le
\|f'\|_{L^\infty}  \int_{(0,x_1^*)\times I}|\partial_1 u_1-1|\, dx
\le c \frac{L^{1/2}E^{1/2}}{H^{1/2}}\,.
\end{equation*}
Using (\ref{eqDuFI}) we obtain also in this case  (\ref{eqk2alletermsad}),
which 
concludes the proof. 
\end{proof}

\section*{Acknowledgements}
We thank Barbara Zwicknagl for interesting discussions and comments on this work.
Most of the present results are part of the PhD thesis \cite{B9_ChanDiss}.
The first author would like to thank Ruth Joachimi for her encouraging support. 
This work was partially supported by the Deutsche Forschungsgemeinschaft
through SFB 611, Project B9.
\def\polhk#1{\setbox0=\hbox{#1}{\ooalign{\hidewidth
  \lower1.5ex\hbox{`}\hidewidth\crcr\unhbox0}}}

\end{document}